\documentclass[10pt]{article}
\usepackage{amssymb,amsmath,amsthm}

\newcommand{\R}{\mathbb{R}}
\newcommand{\N}{\mathbb{N}}
\newcommand{\F}{\mathbb{F}}
\newcommand{\C}{\mathbb{C}}
\DeclareMathOperator{\re}{Re}
\newcommand{\de}{\partial}

\def\half{{1\over 2}}
\def\abs#1{|#1|}

\def\norm#1{\|#1\|}
\def\cat{\mathop{\rm cat}\nolimits}
\def\dist{\mathop{\rm dist}\,}

\def\wlimit{\rightharpoonup}

\def\supp{\mathop{\rm supp}}

\def\be{\begin{equation}}
\def\ee{\end{equation}}

\def\epsilon{\varepsilon}
\def\intRN{\int_{\R^N}}

\def\Ie{I_\epsilon}
\def\Je{J_\epsilon}
\def\Qe{Q_\epsilon}

\def\calS{{\cal S}}

\def\calV{{\cal V}}
\def\calK{{\cal K}}

\def\whS{\widehat S}

\def\wrho{\widehat\rho}

\def\ue{u_\epsilon}

\def\cuplength{{\rm cupl}}

\def\hdelta{\hat\delta}

\def\calXed{{\cal N}_{\epsilon,\delta}}
\def\calXp{\calXed^{E(m_0)+\hdelta}}
\def\calXm{\calXed^{E(m_0)-\hdelta}}

\newcommand{\e}{\varepsilon}

\newtheorem{theorem}{Theorem}[section]
\newtheorem{proposition}[theorem]{Proposition}
\newtheorem{corollary}[theorem]{Corollary}
\newtheorem{lemma}[theorem]{Lemma}
\newtheorem{remark}[theorem]{Remark}
\newtheorem{definition}[theorem]{Definition}

\begin{document}

\title
{Multiple complex-valued solutions for nonlinear magnetic Schr\"odinger equations \\ \vskip 0.5 true cm
{\large Dedicated to Paul Rabinowitz with great admiration and esteem}\vskip 0.5 true cm}

\author{Silvia Cingolani \thanks{The first author is partially supported by GNAMPA-INDAM Project 2015
	 	\emph{Analisi variazionale di modelli fisici non lineari}.}  
	 \\
Dipartimento di Meccanica, Matematica e Management\\
Politecnico di Bari\\
Via E. Orabona 4, 70125 Bari, Italy\\
\\
Louis Jeanjean \thanks{This work has been carried out in the framework of the project NONLOCAL (ANR-14-CE25-0013), funded by the French National Research Agency (ANR). }
\\
Laboratoire de Math\'ematiques (UMR CNRS 6623)\\
Universit\'e de Franche-Comt\'e\\
16, Route de Gray, 25030 Besan\c{c}on Cedex, France\\
\\
Kazunaga Tanaka \thanks{The third author is partially supported by JSPS Grants-in-Aid for Scientific Research (B) (25287025)
and Waseda University Grant for Special Research Projects 2015B-189.}
\\
Department of Mathematics\\
School of Science and Engineering\\
Waseda University\\
3-4-1 Ohkubo, Shijuku-ku, Tokyo 169-8555, Japan}

\date{}

\maketitle


\begin{abstract}
We study, in the semiclassical limit, the singularly perturbed nonlinear Schr\"odinger equations
\be \label{eq:0.1}
L^{\hbar}_{A,V} u = f(|u|^2)u \quad\hbox{in $\R^N$}
\end{equation}
where $N \geq 3$, $L^{\hbar}_{A,V}$ is the Schr\"odinger operator  with a magnetic field having 
source in a $C^1$ vector potential $A$ and a scalar continuous (electric) potential $V$ defined by
\begin{equation}
L^{\hbar}_{A,V}=
-\hbar^2 \Delta-\frac{2\hbar}{i} A \cdot
\nabla + |A|^2- \frac{\hbar}{i}\operatorname{div}A + V(x).
\end{equation}
Here $f$ is a nonlinear term which satisfies the so-called Berestycki-Lions conditions.
We assume that there exists a bounded domain $\Omega \subset \R^N$ such that
\[ m_0 \equiv \inf_{x \in \Omega} V(x) <    \inf_{x \in \partial \Omega} V(x) \]
and we set $K = \{ x \in \Omega \ | \ V(x) = m_0\}$. For $\hbar >0$ small we prove the existence of 
at least ${\cuplength}(K) + 1$
geometrically distinct, complex-valued solutions to (\ref{eq:0.1}) whose modula concentrate 
around $K$ as $\hbar \to 0$.
\end{abstract}


\setcounter{equation}{0}
\section{\label{section:1} Introduction}

In the present work we study, in a semiclassical regime, a nonlinear
magnetic Schr\"o\-dinger equation, which arises
in many fields of physics, in particular condensed matter physics and
nonlinear optics.  More precisely, we are looking
for stationary states to the evolution equation
\begin{equation}\label{eq:evol}
i\hbar \frac{\de \psi}{\de t} = L^{\hbar}_{A,W} \psi - f(|\psi|^2) \psi
\quad\hbox{in $\R^{+}\times\R^N$},
\end{equation}
where $i^2=-1$ and
$L^{\hbar}_{A,W}$ is the Schr\"odinger operator  with a magnetic field $B$ having source in 
a $C^1$ vector potential $A$ and a continuous scalar (electric) potential $W$ defined by
\begin{equation}
L^{\hbar}_{A,W}= \left( \frac{\hbar}{i} \nabla -A\right)^2 + W(x) =
-\hbar^2 \Delta-\frac{2\hbar}{i} A \cdot
\nabla + |A|^2- \frac{\hbar}{i}\operatorname{div}A + W(x).
\end{equation}
Mathematically the transition from Quantum mechanics to Classical mechanics is described letting 
to zero the Planck constant ($\hbar \to 0$) and the solutions, which exist for small value of $\hbar$, 
are  usually referred  {\it
semiclassical bound states}.
For the physical background, we refer to \cite{sulem}.

The {\sl ansatz} that the solution $\psi(x,t)$
to $(\ref{eq:evol})$ is a standing wave of the form
\[
\psi(t,x)=e^{-i E\hbar^{-1}t} u(x),
\]
with $E \in \R$ and $u \colon \R^N \to \C$, leads us to solve
the complex-valued semilinear elliptic equation
\begin{equation}\label{eq:SS}
L^{\hbar}_{A,W} u = Eu + f(|u|^2) u \quad\hbox{in $\R^N$}.
\end{equation}
In the work we consider an electric potential $W(x)$ which is
bounded from below on $\R^N$, and we choose $E$ such that $V(x)=
W(x) -E$ is strictly positive. Hence equation $(\ref{eq:SS})$
becomes
\begin{equation}\label{eq:S}
L^{\hbar}_{A,V} u = f(|u|^2)u \quad\hbox{in $\R^N$}
\end{equation}
where $V$ is a strictly positive potential.

Concerning the zero external magnetic case, there is an extensive literature starting from the paper 
by Floer and Weinstein \cite{FlWe} (see for instance  \cite{ABC,ams,Cila,CL,DF,DFbis,JT2,Li,O,R,BJ,BJT,BT1,CJT}).

To our knowledge, the first paper, in which 
equation $(\ref{eq:S})$ is considered, is
due to Esteban and Lions \cite{EL}. The authors proved the existence of
standing wave solutions to $(\ref{eq:S})$ for $\hbar>0$ fixed and $N=3$, by a constrained
minimization for a constant electric potential and a cubic nonlinearity.
Concentration and compactness arguments are applied to solve such
minimization problems for special classes of magnetic fields.

Successively, in \cite{Ku} Kurata proved the existence of least
energy solutions to $(\ref{eq:S})$ for any fixed $\hbar >0$, under
some assumptions linking the magnetic field $B$ and the electric
potential $V$.
He also studied the concentration phenomena of a family of least solutions of $(\ref{eq:S})$ in 
the semiclassical limit, showing that concentration of the modula of such solutions occurs at
global minima of the electric potential $V$. For periodic scalar and vector potentials,
we refer to the paper by  Arioli and Szulkin \cite{ASz}.

A first multiplicity result of semiclassical solutions to
$(\ref{eq:S})$, when $f(t) = |t|^{(p-1)/2}$, $1 < p < (N+2)/(N-2)$ for $N>2$ and $ 1 < p <  + \infty$ if $N=2$,
as $\hbar \to 0$, has been proved by the first author in
\cite{C}.
Since the nonlinearity satisfies the Nehari monotonicity condition, the problem can be reduced to the search 
of critical points of a functional constrained to a Nehari manifold and the number of complex-valued solutions 
to $(\ref{eq:S})$ can be related to the topology of (global) sublevel of the functional  
by standard deformation theorems for Hilbert manifold without boundary. Finally by means of an entrance map 
and a barycenter map, this number is estimated by the topological richness of the set of global minima of 
the electric potential $V$. 

\vspace{2mm}
The existence of complex-valued solutions of the magnetic Schr\"odinger equations in $\R^3$,
whose modula concentrate at local minima of $V$, has been derived in \cite{CS1} using a penalization argument 
(see also \cite{CS}) for a cubic nonlinearity. 
See also \cite{AlFiFu} for an extention of the results of \cite{CS,CS1}.
Successively, in \cite{CJS} the existence of multi-peaks solutions for nonlinear Schr\"odinger equations  
with an external magnetic field is proved for a more general nonlinearity which satisfies the Berestycki-Lions 
conditions, using a variational approach developed in \cite{BJ} for the zero external magnetic case. Recently,  
the existence of semiclassical cylindrically symmetric solutions, whose moduli concentrate around circles driven by the magnetic and electric fields, has been established in  \cite{BCN}, assuming that $A$ and $V$ are cylindrically symmetric potentials.

In this paper we are interested to find a multiplicity result to $(\ref{eq:S})$ for a large class of 
magnetic potentials $A$ and under nearly optimal assumptions on the nonlinearity.

\vspace{2mm}
Throughout the paper, we assume that $N \geq 3$ and
\begin{itemize}
\item[(A1)] $A\colon \R^N\to\R^N$ is a $C^1$ function such that, for some positive constants $C, \gamma$ 
$$|J_A(x)| \leq C e^{\gamma |x|}, \quad  \forall x \in \R^N$$ 
where $J_A$ denotes the Jacobian matrix of $A$ at $x$;
\item[(V1)] $ V \in C(\R^N, \R)$ and $\inf_{x \in \R^N} V(x) = \underline{V} >0$;
\item[(V2)] There is a bounded domain $\Omega \subset \R^N$ such that
\[ m_0 = \inf_{x \in \Omega} V(x) <    \inf_{x \in \partial \Omega} V(x).\]
\end{itemize}

On the nonlinearity we require that
\begin{itemize}
\item[(f1)] $f \colon [0, + \infty) \to \R$ is continuous;
\item[(f2)] $f(0)=\lim_{\xi\to 0^+}{f(\xi)}=0$;
\item[(f3)]
there  exists some $1 < p < \frac{N+2}{N-2},$ 
                 such that  $\lim_{\xi \to \infty} f(\xi^2)/\xi^{p-1}=0$;
\item[(f4)] There exists $\xi_0>0$ such that
    $$ 
    F(|\xi_0|^2) > m_0 \xi_0^2 \quad \mbox{where} \quad F(\xi)=\int_0^\xi f(\tau)\, d \tau.
    $$
    \end{itemize}

We remark that, under our assumptions of $f$, the search of solutions to $(\ref{eq:S})$ cannot be reduced 
to the study of the critical points of a functional restricted to a Nehari manifold.
We obtain multiplicity results of complex-valued solutions to $(\ref{eq:S})$ by means of a new variational 
approach developed in the recent paper \cite{CJT} for the scalar nonlinear Schr\"odinger equation corresponding 
to $A=0$. Precisely, we seek for critical points of the indefinite Euler functional associated to
problem $(\ref{eq:S})$ in a suitable neighborhood of expected solutions, 
suggested
by a complex-valued limiting problem.
In order to apply critical point theory, one of the main problem in this approach is the presence of the boundary.
However
we  recognize that this neighborhood is positively invariant for a pseudo gradient flow and we  derive a deformation 
theorem. In this way we relate the number of complex-valued solutions to the relative category of two sublevels 
of the functional in the neighborhood of expected solutions.  Finally this relative category will be estimated 
by means of the cuplength of the local minima set $K$ of $V$.  To this aim, we need to construct  two maps between 
topological pairs, which involve a barycenter map and a Pohozaev type function.
We remark that the presence of the magnetic field produces several additional difficulties, due to the fact 
that  the space $H_{\e, V, A}(\R^N, \C)$, where the Euler functional associated to  $(\ref{eq:S})$ is naturally 
defined,  can not be embedded into $H^1(\R^N, \C)$ where the limiting problem (as $\hbar \to 0$) is set up 
(see \cite{EL}).  In order to perform our variational arguments, we crucially use of the diamagnetic inequality 
which, in particular allows to prove that 
the least energy level of the complex-valued limiting problem  coincides with the least energy level of a 
real-valued problem (see Section 2). This fact enables to construct a barycenter map and a Pohozaev type map, 
which act only on the modula of the complex-valued functions.

Our main result is the following:

\vspace{1mm}

\begin{theorem}\label{claim:1.1}
Suppose $N\geq 3$ and that (A1), (V1)--(V2) and (f1)--(f4) hold. Assume in addition that 
$\sup_{x\in\R^N}V(x)<\infty$. Then letting
$K=\{ x\in\Omega ; \, V(x)=m_0\}$, for sufficiently small $\epsilon >0$, $(\ref{eq:S})$ has at least
$\cuplength(K)+1$ geometrically distinct, complex-valued solutions $v_{\e}^i$, $i=1, \dots, \cuplength(K)+1$, 
whose modula concentrate as $\e \to 0 $ in $K$,
where $\cuplength(K)$ denotes the cup-length defined with Alexander-Spanier cohomology with coefficients
in the field $\F$.
\end{theorem}

We say that two complex-valued solutions to $(\ref{eq:S})$ are geometrically distinct if their
$\mathbb{S}^{1}$-orbits are different.
Our theorem covers the relevant physical case of constant magnetic fields $B$ which leads to vector
potentials $A$, having a polynomial growth on $\R^N$.  For
instance, if $B$ is the constant magnetic field $(0,0,b)$, then a
suitable vector field is given by $A(x)= \frac{b}{2}(-x_2,x_1,0)$.
In physical literature the potential $A$ corresponds to the
so-called {\it Lorentz gauge} (see \cite{EL}).

\begin{remark}\label{examples}
If $K=S^{N-1}$, the $N-1$ dimensional sphere in $\R^N$, then $\cuplength(K) + 1=  \cat(K) =2$.  If $K=T^N$ is 
the $N$-dimensional torus, then $\cuplength(K) + 1 = \cat(K)= N + 1$.
However in general $\cuplength(K) +1 \leq \cat(K)$.
\end{remark}

\begin{remark}\label{mean-concen}
When we say that the solutions  $v_{\e}^i$, $i=1, \dots, \cuplength(K)+1$ of Theorem \ref{claim:1.1} concentrate 
when $\e \to 0$ in K, we mean that there exists a maximum point $x_{\e}^i$ of $|v_{\e}^i|$ such that 
$\lim_{\e \to 0}dist(x_{\e}^i, K) =0$ and that, for any such $x_{\e}^i$, $|w_{\e}^i| = |v_{\e}^i(\e(\cdot + x_{\e}^i))|$ converges, up to a subsequence, uniformly to a least energy solution of
    $$- \Delta U + m_0 U = f(U), \quad U >0, \quad U \in H^{1}(\R^N, \R).
    $$
We also have
    $$  |v_{\e}^i(x)| \leq C exp(- \frac{c}{\e}|x- x_{\e}^i|) \quad \mbox{ for some } c,\, C >0.
    $$
\end{remark}

\begin{remark} In addition to condition $(V1)$ the boundedness of $V$ from above is assumed in 
Theorem \ref{claim:1.1}. Arguing as in \cite{BJ,BT1,CJS}  we could prove Theorem \ref{claim:1.1} 
without this additional assumption. However, for the sake of simplicity, we assume here the boundedness of $V$.
\end{remark}

\setcounter{equation}{0}
\section{\label{section:2} The variational framework $H_{\e, V, A}$}

Let $\hbar =
\varepsilon$,   $v(x) = u(\e x)$, $A_\e
(x) = A(\e x)$ and $V_\e(x) = V(\e x).$ Then equation $(\ref{eq:S})$ is
equivalent to
\begin{equation}\label{eq:2.1}
\left( \frac{1}{i}\nabla -A_\e(x)\right)^2 v + V_\e(x)v - f(|v|^2)v
= 0, \quad x \in \R^N.
\end{equation}

Let $H_{\e,V,A} (\R^N, \C)$ be the Hilbert space defined by the completion of
$C_0^\infty(\R^N, \C)$ under the scalar product
\begin{equation} \label{eq:2.2}
\langle u,v \rangle_{\e,V,A} = \re \int_{\R^N} \left(\frac{1}{i}\nabla u
- A_\e (x)u \right) \left(\overline{\frac{1}{i}\nabla v-
A_\e (x)v}\right)  + V_\e(x) u \overline{v} \ dx .
\end{equation}
We denote by $\| \cdot  \|_{\e, V, A}$ the associated norm and in the special case $V=1$, we set 
$\| \cdot  \|_{\e} := \| \cdot  \|_{\e, 1, A}$. Also let $H_\e$ denote the space $H_{\e,1,A} (\R^N, \C)$.

In what follows we use the notations:
    \begin{eqnarray*}
        \norm u_{H^1} &=& \left(\intRN \abs{\nabla u}^2+u^2\, dx\right)^{1/2},\\
        \norm u_r &=& \left(\intRN \abs{u}^r\, dx\right)^{1/r} \quad
                        \hbox{for} \ r\in [1,\infty).
    \end{eqnarray*}

It is well known that, in general, there is
no relationship between the spaces $H_\e$ and $H^1(\R^N, \C)$,
namely $H_\e \not \subset H^1(\R^N, \C)$ nor $H^1(\R^N, \C)\not \subset H_\e$ (see \cite{EL}).

We now recall 
the following \textit{diamagnetic inequality}: for every $u \in H_\e$,
\begin{equation}\label{eq:2.3}
\left| \left(\frac{\nabla}{i} - A_\e \right) u\right| \geq |\nabla |u\|,
\quad \hbox{a.e. in $\R^N$. }
\end{equation}
See \cite{EL} for a proof. 
As a consequence of \eqref{eq:2.3},
$|u| \in H^1(\R^N, \R)$ for any $u \in H_{\e}$.
We also have for any $r\in [2,\frac{2N}{N-2}]$ there exists $C_r>0$ independent of $\e$ such that
    \begin{equation}\label{eq:2.4}  
    \| u\|_r \leq C_r\| u\|_\e \quad \text{for all}\ u\in H_\e.
    \end{equation}
We also note that for any compact set $K\subset\R^N$ there exists a $C_K>0$ independent of $\e\in (0,1]$
such that
    \begin{equation}\label{eq:2.5}
    \| u\|_{H^1(K)} \leq C_K\| u\|_\e \quad \text{for all}\ u\in H_\e.
    \end{equation}
See \cite[Corollary 2.2]{CJS}.
Thus convergence in $H_\e$ implies convergence in $H^1_{loc}(\R^N,\C)$.
We can also derive the following lemma.
\medskip

\begin{lemma}\label{continuity}
For each $\e>0$, 
	$$	H_\e\to H^1(\R^N,\R);\, u\mapsto |u|
	$$
is a continuous map.
\end{lemma}

\begin{proof}
It suffices to show for any strongly convergent sequence $u_n\to u$ in $H_\e$ that $|u_n|\to |u|$ 
strongly in $H^1(\R^N,\R)$.  
Suppose that $(u_n) \subset H_\e$ converges to $u \in H_\e$ strongly in $H_\e$.  We infer, by \eqref{eq:2.4}
that $(u_n)$ strongly converges to $u$ in $L^2 (\R^N, \C)$ and in particular $|u_n| \to |u|$ in  $L^2 (\R^N, \R)$. 
Now from  $(\ref{eq:2.3})$ we deduce that  $(|u_n|)$ is bounded in $H^1(\R^N,\R)$ and weakly converges
in $H^1(\R^N,\R)$, up to subsequences.  Moreover 
since strong convergence in $H_\e$ implies strong convergence in $H^1_{loc}(\R^N)$,
$(|\nabla |u_n|(x)|)$ converges to $|\nabla |u|(x)|$, a.e in $\R^N$. Using again $(\ref{eq:2.3})$, we derive, 
by Lebesgue's Theorem, 
that $\| |u_n| \|_{H^1(\R^N, \R)}$ converges to $ \| |u| \|_{H^1(\R^N, \R)}$ and the conclusion of lemma holds. 
\end{proof}

Another continuity property of $u\mapsto |u|;\, H_\e\to H^1(\R^N)$ will be given in Lemma \ref{claim:2.4} below.

\bigskip
Finally we consider the functional $\Ie$ defined on $H_\e$ by 
    $$
    \Ie(u)=
    \half\int_{\R^N} \left|\left(\frac{1}{i}\nabla  -  A_\e(x) \right)\, u \right|^2 \, dx  
    +\half\int_{\R^N} V_\e(x) |u|^2\, dx
        - \half\intRN F(|u|^2)\, dx,
    $$
where $F(s)= \int_0^s f(t) \, dt$. Without loss of generality, we may  assume that $f(\xi)=0$ for all 
$\xi\leq 0$. It is standard that the functional is $C^1$ and its critical points are solutions of  \eqref{eq:2.1}.

\smallskip
We notice that the group $\mathbb{S}^{1}$ of unit complex numbers acts on $H_{\varepsilon}$ 
by scalar multiplication
$(\gamma,u)\mapsto\gamma u.$ This action is unitary, that is,
\[
\left\langle \gamma u,\gamma v\right\rangle _{\varepsilon,V,A}=\left\langle
u,v\right\rangle _{\varepsilon,V,A}\text{ \ \ \ \ }\forall\gamma\in
\mathbb{S}^{1},\text{ }u,v\in H_{\varepsilon}.
\]
Since the functional $I_{\varepsilon}$ is $\mathbb{S}^{1}$-invariant, that is,%
\begin{align*}
I_{\varepsilon}(\gamma u)  &  =I_{\varepsilon}(u)\text{ \ \ \ \ }%
\forall u\in H_{\varepsilon},\text{
}\gamma\in\mathbb{S}^{1}
\end{align*}
then, if $u \in H_\e$ is a critical point of
$I_{\varepsilon}$, every point $\gamma u$ in the $\mathbb{S}^{1}$-orbit of $u$
is a critical point of $I_{\varepsilon}$.   We say that two critical
points of $I_{\varepsilon}$\ are geometrically distinct if their
$\mathbb{S}^{1}$-orbits are different.
We shall apply
$\mathbb{S}^{1}$-equivariant Lusternik-Schnirelmann theory to obtain a lower bound for the
number of critical $\mathbb{S}^{1}$-orbits of $I_{\varepsilon}$.
\medskip

\subsection{\label{subsection:2.1} Scalar and Complex-valued Limit problems}
Let us consider for $a>0$ the scalar limiting equation of
(\ref{eq:2.1})
\begin{equation}\label{eq:2.6}
-\varDelta u + a  u = f(|u|^2) u, \quad u \in H^1(\R^N, \R).
\end{equation}
(\ref{eq:2.6}) can be obtained as follows:  let $x=y+p/\e$ in (\ref{eq:2.1})
and take a (formal) limit as $\e\to 0$, then we have
    $$  (\frac{1}{i}\nabla -A(p))^2 v+V(p)v= f(|v|^2)v \quad \text{in}\ \R^N.
    $$
Setting $u(x)=e^{-iA(p)x}v(x)$ and considering real-valued solutions, 
we obtain (\ref{eq:2.6}) with $a=V(p)$.
Solutions to $(\ref{eq:2.6})$ correspond to critical points of the 
functional $L_a \colon H^1(\R^N, \R)\to \R$ defined by
\begin{equation}\label{eq:limit}
L_a (u)= \frac12 \int_{\R^N} \big( |\nabla u|^2 + a|u|^2 \big) dy -\half \int_{\R^N}
F(|u|^2) dy.
\end{equation}
We denote by $E(a)$ the least energy level for \eqref{eq:2.6}.
That is,
    $$
    E(a)=\inf\{ L_a(u);\, u\in H^1(\R^N,\R)\setminus \{ 0\},\ L_a'(u)=0\}.
    $$
In \cite{BL} it is proved that there exists a least energy solution of \eqref{eq:2.6}, for any $a >0$, 
if (f1)--(f4) are satisfied (here we consider (f4) with $m_0 = a$). 
Also it is showed that each solution $u$ of \eqref{eq:2.6} satisfies the Pohozaev's identity
\begin{equation}\label{eq:poho}
\frac{N-2}{2}\int_{\R^N} |\nabla u|^2 dx + N \int_{\R^N} a \frac{u^2}{2} -\half F(|u|^2) dx = 0.
\end{equation}
From this we immediately deduce that, for any solution~$\omega$ of
\eqref{eq:2.6},
 \begin{equation} \label{eq:2.9}
\frac{1}{N}\int_{\R^N}|\nabla \omega|^2 dy  = L_a(\omega).
\end{equation}

We also consider the complex-valued equation, for $a>0$,
\begin{equation}\label{eq:2.10}
-\varDelta u + a  u = f(|u|^2) u, \quad u \in H^1(\R^N, \C).
\end{equation}
In turn solutions of $(\ref{eq:2.10})$ correspond to critical points of the
functional $L^{\C}_a :  H^1(\R^N, \C) \to \R$, defined by
\begin{equation}\label{eq:2.11}
L^{\C}_a (v)= \frac12 \int_{\R^N} \left( |\nabla v|^2 + a|v|^2
\right)dy - \half\int_{\R^N} F(|v|^2)dy.
\end{equation}
We denote by $E^{\C}(a)$ the least energy level for \eqref{eq:2.10}, that is
    $$  E^{\C}(a)=\inf\{ L^{\C}_a(u);\, u\not=0,\ {L^{\C}_a}'(u)=0\}.
    $$
In \cite{St} the Pohozaev's identity (\ref{eq:poho}) is established for complex-valued solutions of
(\ref{eq:2.10}) and thus the equivalent of (\ref{eq:2.9}) holds for such solutions.

In \cite[Lemma 2.3]{CJS}, it has been proved that the least energy
levels of (\ref{eq:2.6}) and (\ref{eq:2.10}) coincide and that any least energy solution $U$ of (\ref{eq:2.10}) 
has the form $e^{i \tau}
\omega$ where $\omega$ is a positive least energy solution of (\ref{eq:2.6}) and $\tau
\in \R$.  \medskip

Now we introduce the notation
$$
\Omega(I) =\{ y\in\Omega;\, V(y)-m_0\in I\}
$$
for an interval $I\subset [0,\inf_{x\in\partial\Omega}V(x)-m_0)$.
We choose a small $\nu_0>0$ such that

\begin{itemize}
\item[(i)] $0<\nu_0<\inf_{x\in\partial \Omega}V(x) -m_0$;
\item[(ii)] $F(|\xi_0|^2)>\half(m_0+\nu_0)\xi_0^2$;

\item[(iii)] $\Omega([0,\nu_0])\subset K_d$, where $K = \{ x \in \Omega \ | \ V(x) = m_0\}$ and $d>0$ is a constant
for which Lemma \ref{claim:5.4} (Section \ref{section:5}) holds.
\end{itemize}

From \cite{BL}  we note that, under our choice of $\nu_0 >0$,  $E(a)$ is attained for $a\in [m_0,m_0+\nu_0]$.  
Clearly
$a\mapsto E(a);\, [m_0,m_0+\nu_0]\to \R$ is continuous and strictly increasing.
Choosing $\nu_0>0$ smaller if necessary, we may assume
    $$
    E(m_0+\nu_0) < 2E(m_0).
    $$
We choose $\ell_0\in (E(m_0+\nu_0), 2E(m_0))$ and we set
    $$  S^{\C}_{a,\ell_0} = \{ U\in H^1(\R^N, \C)\setminus\{ 0\} ; \,
        (L^{\C}_a)'(U)=0,\ L^{\C}_a(U) \leq \ell_0, \ |U(0)|=\max_{x\in\R^N} |U(x)| \}.
    $$
We also define
 $$  \whS_{\nu_0,\ell_0}=\bigcup_{a\in [m_0,m_0+\nu_0]} S^{\C}_{a,\ell_0}.
    $$
Following the proof of  \cite[Proposition 1]{BJ}, we can show that the set
 $\whS_{\nu_0,\ell_0}$ is compact in $H^1(\R^N, \C)$ and that its elements have a uniform exponential
decay. Namely that there exist $C$, $c>0$ such that
\begin{equation}\label{expdecay}
      |U(x)| +  \, \abs{\nabla U(x)} \leq C \exp(-c\abs x) \quad \hbox{for all} \quad  U\in \whS_{l_0,\nu_0}.
\end{equation}
By \cite[Lemma 2.3]{CJS}, each element of $ S^{\C}_{m_0,E(m_0)}$ is of the form $e^{i \tau}\omega$ 
where $\tau \in \R$ and $\omega$ is a real least energy solution of \eqref{eq:2.6}. Thus
    $$  P_0(\omega)=1,
    $$
where $P_0$ is defined as
\begin{equation}\label{eq:2.13}
        P_0(u) 
        =\left( {N\int_{\R^N} \frac{1}{2}F(\abs u^2)-{m_0\over 2}\abs u^2 \,dx
            \over {N-2\over 2}\norm{\nabla \abs u}_2^2}\right)^\half.
\end{equation}
We note that $P_0(\omega(\frac{x}{s}))=s$ and

\begin{lemma}\label{pohozaev}
Suppose that $u\in H^1(\R^N,\C)$ satisifes $P_0(u)\in (0,\sqrt{\frac{N}{N-2}})$.  Then
    $$  L_{m_0}(|u|) \geq g(P_0(u)) E(m_0),
    $$
where
    \be\label{eq:2.14}
    g(t)=\frac{1}{2} (Nt^{N-2}-(N-2)t^N).
    \ee
\end{lemma}

\begin{proof}
By the scaling property
    $$  L_{m_0}(|u(\frac{x}{s})|) = \frac{s^{N-2}}{2}\| \nabla|u|\|_2^2
        +s^N\left( \frac{m_0}{2}\| u\|_2^2 -\frac{1}{2} \int_{\R^N} F(|u|^2)\, dx\right)
    $$
and the characterization of $E(m_0)$
    $$  E(m_0)=\inf \{ L_{m_0}(u);\, u\in H^1(\R^N,\R)\setminus\{ 0\},\, P_0(u)=1\},
    $$
we can deduce the conclusion of Lemma \ref{pohozaev}.  See \cite{JT1} and \cite[Lemma 2.1]{CJT}.
\end{proof}

We claim that, as $\ell\to E(m_0)$ and $\nu\to 0$, $\widehat S_{\nu,\ell}$ shrinks to 
$S_{m_0, E(m_0)}^{\C}$ in $H^1(\R^N,\C)$. More precisely we have
\begin{equation}\label{eq:2.15}
\displaystyle
    \lim_{\nu \to 0, \ell \to E(m_0)} \sup_{{\widetilde U} \in \whS_{\nu,\ell}} 
    \inf_{U \in S_{m_0,E(m_0)}^{\C}} ||U-{\widetilde U}||_{H^1} =0.
\end{equation}
In fact, suppose $\nu_n>0$, $\ell_n>E(m_0)$ and $U_n\in \whS_{\nu_n,\ell_n}$ satisfy
$\nu_n\to 0$ and $\ell_n\to E(m_0)$, then by the compactness of $\whS_{\nu,\ell}$ for each
$\nu\geq 0$, $\ell\geq E(m_0)$, $U_n$ converges to some $U\in S_{m_0,E(m_0)}^{\C}$ in $H^1(\R^N,\C)$.
Thus \eqref{eq:2.15} holds.


As a consequence of (\ref{eq:2.15}) for $\ell_0$ close to $E(m_0)$ and $\nu_0 >0$ small, we have
\begin{equation}\label{estP}
    P_0(U) \in \Big(\half,\sqrt{N\over N-1}\Big) \quad \hbox{for all}\ U\in \widehat S_{\nu_0, \ell_0}.
\end{equation}
We fix such $\ell_0$ and $\nu_0$ and we write $\widehat S=\widehat S_{\nu_0,\ell_0}$. \medskip

In what follows, we try to find our
critical points in the following bounded subsets of $H_\e$:
    $$  \calS_\e(r)= \{ e^{i A(\e p)(x-p)} U(x-p) + \varphi(x);\, \e p \in \overline \Omega,  
        \ U\in\whS,\ \norm\varphi_{\e}<r\}
    $$
for a $r>0$.

\subsection{\label{subsection:2.2} A  Pohozaev map in $\calS_\e(r)$}
First we give an equi-continuity result of $u\mapsto |u|;\, H_\e\to H^1(\R^N,\R)$.

\begin{lemma}\label{claim:2.4}
For any $r>0$ there exists $r_{**}>0$ such that for small $\e>0$
    \begin{equation}\label{eq:2.17}
    \| |u(x)| -|U(x-p)|\|_{H^1} <r
    \end{equation}
for any $u(x)=e^{iA(\e p)(x-p)}U(x-p)+\varphi(x)\in \calS_\e(r_{**})$ with $\e p\in\overline\Omega$,
$U\in\widehat S$, $\|\varphi\|_\e<r_{**}$.
\end{lemma}

\begin{proof}
It suffices to show
    \begin{eqnarray}
    && \ \| |u_n(x)| -|U_n(x-p_n)|\|_2 \to 0, \label{eq:2.18}\\
    && \ \| \nabla|u_n(x)| -\nabla|U_n(x-p_n)|\|_2 \to 0 \label{eq:2.19}
    \end{eqnarray}
for $u_n=e^{iA(\e_n p_n)(x-p_n)}U_n(x-p_n)+\varphi_n(x)\in \calS_{\e_n}(r_n)$ with
$\e_n\to 0$, $r_n\to 0$, $\e_n p_n\in \overline\Omega$, $U_n\in \widehat S$ and
    \begin{equation}\label{eq:2.20}
    \| \varphi_n\|_{\e_n} <r_n\to 0.
    \end{equation}
Since $\widehat S$ is compact in $H^1(\R^N,\C)$, we may assume that $U_n\to U_0\in \widehat S$
and $\e_n p_n\to p_0\in \overline\Omega$ as $n\to \infty$.  We proceed in several steps.

\smallskip

\noindent
{\sl Step 1: \eqref{eq:2.18} holds.}

\smallskip

\noindent
By \eqref{eq:2.4} and \eqref{eq:2.20},
    $$  \| u_n(x)-e^{iA(\e_n p_n)(x-p_n)}U_n(x-p_n)\|_2 = \|\varphi_n\|_2 \leq \|\varphi_n\|_{\e_n} \to 0.
    $$
Since $u\mapsto |u|;\, L^2(\R^N,\C)\to L^2(\R^N,\R)$ is continuous, we have
    $$  \| |u_n(x)| -|U_n(x-p_n)|\|_2 = \| |u_n(x)|-|e^{iA(\e_n p_n)(x-p_n)}U_n(x-p_n)|\|_2 \to 0.
    $$

\smallskip

\noindent
{\sl Step 2: $\| e^{iA(\e_n p_n)(x-p_n)}(U_n(x-p_n)-U_0(x-p_n))+\varphi_n(x)\|_{\e_n}\to 0$.}

\smallskip

\noindent
Observe that
    \begin{eqnarray*}
    &&\| \left(\frac{1}{i}\nabla-A(\e_nx)\right) (e^{iA(\e_n p_n)(x-p_n)}(U_n(x-p_n)-U_0(x-p_n))) \|_2 \\
    &=& \| \left(\frac{1}{i}\nabla-A(\e_nx+\e_n p_n)\right) (e^{iA(\e_n p_n)x}(U_n(x)-U_0(x))) \|_2 \\
    &=& \| A(\e_n p_n)e^{iA(\e_n p_n)x} (U_n-U_0)
        + \frac{1}{i} e^{iA(\e_n p_n)x}(\nabla U_n-\nabla U_0)\\
    && \ - A(\e_n x+\e_n p_n) e^{iA(\e_n p_n)x}(U_n-U_0)\|_2 \\
    &=& \| (A(\e_n p_n)-A(\e_n x+\e_n p_n))(U_n-U_0)+\frac{1}{i} (\nabla U_n-\nabla U_0)\|_2 \to 0 \quad \text{as}\ n\to \infty.
    \end{eqnarray*}
Here we have use the fact that  $U_n\to U_0$ in $H^1(\R^N, \C)$ and that the elements in $\widehat S$ have
a uniform exponential decay. Clearly also 
    $$  \| e^{iA(\e_n p_n)(x-p_n)}(U_n(x-p_n)-U_0(x-p_n))\|_2 
        = \| U_n-U_0\|_2 \to 0.
    $$
By \eqref{eq:2.20}, we have the conclusion of Step 2.

\smallskip

\noindent
{\sl Step 3: $|e^{iA(\e_n p_n)(x)}U_n(x)+\varphi_n(x+p_n)| \to |U_0(x)|$ in $H^1_{loc}(\R^N,\R)$.\\
In particular, after taking a subsequence 
    $$  \nabla |e^{iA(\e_n p_n)(x)}U_n(x)+\varphi_n(x+p_n)| \to \nabla|U_0(x)|\quad \text{a.e. in}\ \R^N.
    $$
}

\smallskip

\noindent
Using notation
    $$  \|v\|_{\e_n,1,A(\cdot+\e_n p_n)}^2 = \| (\frac{1}{i}\nabla -A(\e_nx+\e_n p_n))v\|_2^2+\| v\|_2^2,
    $$
we have by Step 2
    \begin{eqnarray}    
        &&\| e^{iA(\e_n p_n)x}(U_n(x)-U_0(x))+\varphi_n(x+p_n)\|_{\e_n,1,A(\cdot+\e_n p_n)} \nonumber\\
        &=& \| e^{iA(\e_n p_n)(x-p_n)}(U_n(x-p_n)-U_0(x-p_n))+\varphi_n(x)\|_{\e_n} \nonumber\\
        &\to& 0 \qquad \text{as} \ n\to \infty. \label{eq:2.21}
    \end{eqnarray}
As in \eqref{eq:2.4}, we get
    $$  e^{iA(\e_n p_n)x}(U_n(x)-U_0(x))+\varphi_n(x+p_n) \to 0 \quad \text{in} \ H^1_{loc}(\R^N,\C),
    $$
from which the conclusion of Step 3 follows.

\smallskip

\noindent
{\sl Step 4: \eqref{eq:2.19} holds.}

\smallskip

\noindent
By \eqref{eq:2.21},
    $$  \| \left(\frac{1}{i}\nabla -A(\e_n x+\e_n p_n)\right) (e^{iA(\e_n p_n)x}(U_n(x)-U_0(x))+\varphi_n(x+p_n))\|_2
        \to 0.
    $$
Thus $\left(\frac{1}{i}\nabla -A(\e_n x+\e_n p_n)\right) (e^{iA(\e_n p_n)x}U_n(x)+\varphi_n(x+p_n))$ converges to\\
$\left(\frac{1}{i}\nabla -A(p_0)\right) (e^{iA(p_0)x}U_0(x))$ in $L^2(\R,\C)$.  
Therefore, there exists a $h(x)\in L^2(\R^N)$  such that
    $$  \left| \left(\frac{1}{i}\nabla -A(\e_n x+\e_n p_n)\right) (e^{iA(\e_n p_n)x}U_n(x)+\varphi_n(x+p_n)) \right| 
        \leq h(x).
    $$
By the diamagnetic inequality, we have
    \begin{eqnarray*}
    &&\ \left| \nabla |e^{iA(\e_n p_n)x}U_n(x)+\varphi_n(x+p_n)) | 
        -\nabla |U_0(x)|\right| \\
    &\leq& \left| \left(\frac{1}{i}\nabla -A(\e_n x+\e_n p_n)\right) (e^{iA(\e_n p_n)x}U_n(x)+\varphi_n(x+p_n)) \right| 
        + |\nabla U_0(x)| \\
    &\leq& h(x)+ |\nabla U_0(x)| \in L^2(\R^N).
    \end{eqnarray*}
Therefore, by Lebesgue theorem, it follows from Step 3 that
    \begin{eqnarray*}
    \| \nabla|u_n(x+p_n)|-\nabla|U_0(x)|\|_2 
    &=& \| \nabla| e^{iA(\e_n p_n)x}U_n(x)+\varphi_n(x+p_n)) | - \nabla|U_0(x)|\|_2 \\
    &\to & 0.
    \end{eqnarray*}
which is nothing but \eqref{eq:2.19}.
\end{proof}

For a later use we have the following

\begin{corollary}\label{pozz}
There exists $r_0>0$ such that for $\epsilon>0$ small, $P_0:\, {\cal S}_\epsilon(r_0) \to \R$ is 
continuous and 
\begin{equation}\label{estfP}
    P_0(u)\in \Big(0,\sqrt{N\over N-1}\Big) \quad \hbox{for all} \ u\in {\cal S}_\epsilon(r_0).
\end{equation}
\end{corollary}

\begin{proof} 
Since $\| |u|-|U|\|_{H^1}\to 0$ implies $\int_{\R^N} F(|u|^2)\, dx\to \int_{\R^N} F(|U|^2)\, dx$,
$\| u\|_2^2 \to \| U\|_2^2$ and $\|\nabla|u|\|_2^2 \to \|\nabla|U|\|_2^2$,
by Lemma \ref{claim:2.4} there exists  $r_0>0$ such that \eqref{estP} holds for $\epsilon>0$ small.
The continuity of $P_0:\, \calS_\e(r_0)\to \R$ follows from Lemma \ref{continuity}.
\end{proof}

\medskip

\begin{remark}
We remark that there does not exist a constant $C>0$ with the following property:
	$$	\| \nabla|u+\varphi|-\nabla|u| \|_2 \leq C(\| \nabla|\varphi|\|_2+\|\varphi\|_2)
		\quad \text{for all}\ u,\varphi\in H_\e.
	$$
Thus Lemma \ref{claim:2.4} is not a direct consequence from the diamagenetic inequality.
To see such an inequality does not hold, for $n\in\N$ we set $u$, $\varphi_n\in H^1(\R,\C)$ by
	$$	u(x)=\begin{cases}
			1 &\text{if}\ |x|\leq 1, \\
			2-|x| &\text{if}\ |x|\in (1,2],\\
			0 &\text{otherwise},
			\end{cases}
		\quad
		\varphi_n(x)=\begin{cases}
			e^{2\pi i nx} &\text{if}\ |x|\leq 1, \\
			2-|x| &\text{if}\ |x|\in (1,2],\\
			0 &\text{otherwise}.
			\end{cases}
	$$
Then, we have
	$$	\| \nabla|u+\varphi_n| -\nabla|u|\|_{L^2(\R)} \to \infty, \quad
		\sup_{n\in\N} (\| \nabla|\varphi_n|\|_{L^2(\R)} + \|\varphi_n\|_{L^2(\R)}) <\infty.
	$$
We can easily extend this example to general dimension $N$.

\end{remark}

\subsection{\label{subsection:2.3} A barycenter map in $\calS_\e(r)$}

Following \cite{BT1, BT2} we introduce a center of mass in $\calS_\e(r)$.

\begin{lemma}\label{claim:2.6}
There exist $r_0$, $R_0$, $\e_0 >0$,  such that for any $\e \in (0, \e_0)$ there exists a function 
$\Upsilon_\e:\, \calS_\e(r_0)\to\R^N$
such that
    $$  \abs{\Upsilon_\e(u)-p} \leq 2R_0
    $$
for all $u(x)= e^{i A(\e p)(x-p) }U(x-p)+\varphi(x)\in \calS_\e(r_0)$ with $p\in\R^N$,  
$\e p \in \overline\Omega $, $U\in\whS$,
$\|\varphi\|_{\e}\leq r_0$.  Moreover, $\Upsilon_\e$ has the following properties
\begin{itemize}
\item[(i)] $\Upsilon_\e$ is shift equivariant, that is,
    $$  \Upsilon_\e(u(x-y))=\Upsilon_\e(u(x))+y   \quad \text{for all}\  u\in\calS_\e(r_0)\
        \hbox{and}\ y\in\R^N.
    $$
\item[(ii)] $\Upsilon_\e$ is $\mathbb{S}^1$-invariant, that is, 
    $$  \Upsilon_\e(e^{i\tau}u)=\Upsilon_\e(u)   \quad 
        \text{for all}\ u\in\calS_\e(r_0)\ \text{and}\ e^{i\tau}\in \mathbb{S}^1.
    $$
\item[(iii)] 
$\Upsilon_\e:\, \calS_\e(r_0)\subset H_\e\to \R^N$ is a continuous function.  
Moreover, $\Upsilon_\e$ is a locally Lipschitz continuous function of $|u|$ in the following
sense:  $\Upsilon_\e$ satisfies $\Upsilon_\e(u)=\Upsilon_\e(|u|)$ for all $u\in\calS_\e(r_0)$
and there exist constants  $C_1$, $C_2>0$ such that
    \be\label{eq:2.23}  
		\abs{\Upsilon_\e(u)-\Upsilon_\e(v)} \leq C_1\norm{|u|-|v|}_{H^1}
        \quad \text{for all}\   u, v\in\calS_\e(r_0) \  \hbox{with}\
        \norm{|u|-|v|}_{H^1}\leq C_2.
    \ee
\end{itemize}
\end{lemma}

\medskip

\begin{proof}
We set $r_*=\min_{U\in\whS}\| |U| \|_{H^1}>0$ and choose $R_0>1$ such that for $U\in\whS$
    $$  \|\, |U|\, \|_{H^1(\abs x\leq R_0) } > {3\over 4}r_*  \quad \mbox{and} \quad
        \|\, |U|\, \|_{H^1(\abs x\geq R_0)} < {1\over 8}r_*.
    $$
This is possible by the uniform exponential decay (\ref{expdecay}).
For $u\in H^1(\R^N, \C)$ and $q\in\R^N$, we define
    $$  d(q,u)=\psi\left(\inf_{\tilde U\in\whS} \norm{|u(x)|-|\tilde U(x-q)|}_{H^1(\abs{x-q}\leq R_0)}\right),
    $$
where $\psi \in C_0^\infty(\R,\R)$ is such that
    \begin{eqnarray*}
        &&\psi(r)=\begin{cases}    1   &r\in [0,{1\over 4}r_*],\\
                            0   &r\in [\half r_*,\infty),\end{cases}\\
        &&\psi(r)\in [0,1]\quad \hbox{for all}\ r\in [0,\infty).
    \end{eqnarray*}
Now by Lemma \ref{claim:2.4}, there exists $r_{**}\in (0,\frac{1}{8}r_*]$ such that for $\epsilon>0$ small
    \be \label{eq:2.24}
    \|\, |u(x)|-|U(x-p)|\, \|_{H^1} < \frac{1}{8}r_*
    \ee
for $u(x)=e^{iA(\e p)(x-p)}U(x-p)+\varphi(x)\in\calS_\e(r_{**})$.  We set
    $$  \Upsilon_\e(u)={\displaystyle \intRN q\, d(q,u)\, dq
            \over\displaystyle \intRN d(q,u)\, dq}
        \quad \hbox{for}\quad u\in\calS_\e(r_{**}).
    $$
We shall show that  $\Upsilon_\e$ has the desired property.

\vspace{2mm}

Let $u\in \calS_\e(r_{**})$ and write $u(x)=e^{i A(\e p) (x-p)} U(x-p)+\varphi(x)$
($p\in\R^N$, $\e p \in \overline\Omega$, $U\in \whS$, $\norm\varphi_{\e}\leq r_{**}$).

Taking into account that $\norm {|\varphi|}_{H^1} \leq \norm\varphi_{\e}$, we have that for 
$\abs{q-p}\geq 2R_0$ and $\tilde U\in\whS$, we have
    \begin{eqnarray*}
        && \ \norm{|u(x)|-|\tilde U(x-q)|}_{H^1(\abs{x-q}\leq R_0)} \\
        &\geq& \norm{|\tilde U(x-q)| }_{H^1(\abs{x-q}\leq R_0)} - \norm{|u(x)| }_{H^1(\abs{x-q}\leq R_0)}\\
        &\geq& \norm{|\tilde U(x-q)| }_{H^1(\abs{x-q}\leq R_0)} - \norm{|U(x-p)| }_{H^1(\abs{x-q}\leq R_0)}
                -\norm{|u(x)|-|U(x-p)|}_{H^1}\\
        &\geq& \norm{|\tilde U(x-q)|}_{H^1(\abs{x-q} \leq R_0)}
            - \norm{|U(x-p)|}_{H^1(\abs{x-p}\geq R_0)}-{1\over 8}r_*\\
        &>& {3\over 4}r_*-{1\over 8}r_*-{1\over 8}r_*=\half r_*.
    \end{eqnarray*}
Thus $d(q,u)=0$ for $\abs{q-p}\geq 2R_0$.  We can also see that, for small $r>0$
    $$  d(q,u)=1 \quad \hbox{for}\ \abs{q-p}<r.
    $$
Thus $B(p,r)\subset \supp d(\cdot,u)\subset B(p,2R_0)$.
Therefore $\Upsilon_\e(u)$ is well-defined and we have
    $$  \Upsilon_\e(u)\in B(p,2R_0) \quad \hbox{for}\quad  u\in \calS_\e(r_{**}).
    $$
It is clear from the definition that $\Upsilon_\e(u)=\Upsilon_\e(|u|)$ for all $u\in\calS_\e(r_{**})$.
Its shift equivariance, $\mathbb{S}^1$-invariance and locally Lipschitz continuity \eqref{eq:2.23}
can be checked easily.
Thus continuity of $\Upsilon_\e:\, \calS_\e(r_{**})\to \R^N$, where the topology of $\calS_\e(r_{**})$
is induced from $H_\e$, follows from Lemma \ref{continuity}.
\end{proof}

\vspace{3mm}

Using this lemma we have

\begin{lemma}\label{claim:2.7}
There exist $\delta_1>0$, $r_1\in (0,r_0)$ and $\nu_1\in (0,\nu_0)$ such that
for $\epsilon>0$ small
    $$  \Ie(u) \geq E(m_0) +\delta_1
    $$
for all $u\in\calS_\e(r_1)$ with $\epsilon\Upsilon_\e(u)\in\Omega([\nu_1,\nu_0])$.
\end{lemma}

\begin{proof}
We set $\underline M=\inf_{U\in\whS}\norm U_2^2$, $\overline M=\sup_{U\in\whS}\norm U_2^2$. 
It follows from the compactness of $\widehat S$ that $0<\underline M\leq \overline M<\infty$.
For later use in \eqref{eq:2.28} below, we choose $\nu_1\in (0,\nu_0)$ such that
    \be\label{eq:2.25}
        E(m_0+\nu_1)-\half(\nu_0-\nu_1)\overline M > E(m_0).
    \ee
First we claim that for some $\delta_1>0$
    \be\label{eq:2.26}  \inf_{U\in\whS} L_{m_0+\nu_1}^{\C}(U) \geq E(m_0)+3\delta_1.
    \ee
Indeed, on one hand, if $U\in S^{\C}_{a,\ell_0}$ with $a\in [m_0,m_0+\nu_1]$, we have
    \begin{eqnarray*}
        L_{m_0+\nu_1}^{\C}(U) &=& L_a^{\C}(U) +\half(m_0+\nu_1-a)\norm U_2^2\\
        &\geq& E(a) + \half(m_0+\nu_1-a)\underline M
    \end{eqnarray*}
and thus
    \be\label{eq:2.27}
        \inf_{U\in \bigcup_{a\in [m_0,m_0+\nu_1]} S_{a,\ell_0}^{\C}} L_{m_0+\nu_1}^{\C}(U) > E(m_0).
    \ee
On the other hand, if $U\in S^{\C}_{a,\ell_0}$ with $a\in [m_0+\nu_1,m_0+\nu_0]$,
    \begin{eqnarray*}
        L_{m_0+\nu_1}^{\C}(U) &=& L_a^{\C}(U) +\half(m_0+\nu_1-a)\norm U_2^2\\
        &\geq& E(a) -\half(\nu_0-\nu_1)\overline M\\
        &\geq& E(m_0+\nu_1)-\half(\nu_0-\nu_1)\overline M
    \end{eqnarray*}
and using \eqref{eq:2.25}, it follows that
    \be \label{eq:2.28}
        \inf_{U \in\bigcup_{a\in [m_0+\nu_1,m_0+\nu_0]} S^{\C}_{a,\ell_0}} L_{m_0+\nu_1}^{\C}(U) > E(m_0).
    \ee
Choosing $\delta_1 >0$ small enough,
\eqref{eq:2.26} follows from \eqref{eq:2.27} and \eqref{eq:2.28}.
\vspace{2mm}

Now observe that, since elements in $\whS$ have uniform exponential decays,
    $$  \abs{\Ie(e^{i A(\e p)(x- p)}U(x-p))-L_{V(\epsilon p)}^{\C}(U)}\to 0 \quad
        \hbox{as}\ \epsilon\to 0
    $$
uniformly in $U\in\whS$, $\epsilon p\in \overline{\Omega}$.  Thus, by \eqref{eq:2.26}, for $U\in\whS$, 
$\epsilon p\in \Omega([\nu_1,\nu_0])$
    \begin{eqnarray}\label{lowerb}
        \Ie(e^{i A(\e p)(x- p)}U(x-p)) &=& L_{V(\epsilon p)}^{\C}(U) + o(1)
                \geq L_{m_0+\nu_1}^{\C}(U) + o(1) \nonumber\\
        &\geq& E(m_0) +2\delta_1 \qquad \hbox{for $\epsilon>0$ small.}
    \end{eqnarray}
If we suppose that $u(x)=e^{i A(\e p)(x- p)}U(x-p) +\varphi(x)\in\calS_\e(r_0)$ satisfies
$\epsilon\Upsilon_\e(u)\in \Omega([\nu_1,\nu_0])$,  then by Lemma \ref{claim:2.6}, $\epsilon p$
belongs to  a $2\epsilon R_0$-neighborhood of
$\Omega([\nu_1,\nu_0])$. Thus by (\ref{lowerb}) it follows that
    $$  \Ie(e^{i A(\e p)(x- p)}U(x-p)) \geq E(m_0)+{3\over 2}\delta_1 \quad \hbox{for $\epsilon >0$ small.}
    $$
Finally we observe that $I'_\epsilon$ is bounded on bounded sets uniformly in $\epsilon \in (0,1]$ and 
that by the
compactness of $\widehat S,$ $\{ e^{i A(\e p)(x- p)} U(x-p) ;\, U\in\whS,\ \e p \in \overline{\Omega} \}$ is  bounded in $H_\e$. 
Thus choosing
 $r_1\in (0,r_0)$ small, if $u(x) = e^{i A(\e p)(x- p)} U(x-p) + \varphi (x) \in\calS_\e(r_1)$, we have
    $$  \Ie(e^{i A(\e p)(x- p)}U(x-p))+\varphi(x)) \geq \Ie(e^{i A(\e p)(x- p)}U(x-p)) -\half\delta_1 
        \geq E(m_0)+\delta_1.
    $$
Thus, the conclusion of lemma holds. 
\end{proof}

\medskip

\setcounter{equation}{0}
\section{\label{section:3} A penalization on the modulus}
For technical reasons, we introduce a penalized functional $\Je$ following
\cite{BJ}.
Without restriction we can  assume that $\partial\Omega$ is smooth and for $h>0$ we set
    $$  \Omega_h=\{ x\in\R^N\setminus\Omega ; \, \dist(x,\partial\Omega)<h\}\cup\Omega.
    $$
We choose a small $h_0>0$ such that
    $$  V(x)>m_0 \quad \hbox{for all}\ x\in\overline{\Omega_{2h_0}\setminus\Omega}.
    $$
Let
    $$  \Qe(u)=\left(\epsilon^{-2}\norm u_{L^2(\R^N\setminus(\Omega_{2h_0}/\epsilon))}^2
            -1\right)_+^{p+1\over 2}
    $$
and
    $$  \Je(u) = \Ie(u) +\Qe(u).
    $$

Observe that $J_\e$ is $\mathbb{S}^1$-invariant and we say that
two critical points of $J_{\varepsilon}$\ are geometrically distinct if their
$\mathbb{S}^{1}$-orbits are different. In Proposition \ref{claim:3.2} we prove that a $\mathbb{S}^1$ 
critical orbit of $\Je$ is also a $\mathbb{S}^1$ critical orbit of  $\Ie$ for $\e$ small enough. 
Note that the penalization term $\Qe$ forces the concentration of the modula to occur on $\Omega$.
A motivation to introduce $\Je$ is that it  satisfies a useful estimate from below  given in 
Lemma \ref{claim:3.4}.

\medskip

Now we define
    $$  \wrho_\e(u)=\inf\{ \norm{u- e^{i A(\e p) (x-p)}U(x-p)}_{\e};\, p\in\R^N, 
        \e p \in \overline\Omega, \ U\in\whS\}:\, 
        \calS_\e(r_0)\to\R.
    $$
In the following proposition we derive a crucial uniform estimate of $\norm{\Je'}_{(H_{\e})^*}$ in an
annular neighborhood of a set of expected solutions.

\medskip

\begin{proposition} \label{claim:3.1}
There exists $r_2\in (0,r_1)$ with the following property:
for any 
$0<\rho_1<\rho_0 \leq r_2$,
there exists $\delta_2=\delta_2(\rho_0,\rho_1)>0$ such
that for $\epsilon>0$ small
    $$  \norm{\Je'(u)}_{(H_\e)^*} \geq \delta_2
    $$
for all $u\in\calS_\e(r_2)$ with $\Je(u)\leq E(m_0+\nu_1)$ and $(\wrho_\e(u),\epsilon\Upsilon_\e(u))
\in ([0,\rho_0]\times\Omega([0,\nu_0]))\setminus([0,\rho_1]\times\Omega([0,\nu_1]))$.
\end{proposition}

\medskip

\begin{proof}
By (f1)--(f3), for any $a>0$ there exists $C_a>0$ such that
    \begin{equation}\label{eq:3.1}  
    \abs{f(\xi^2)}\leq a +C_a\abs\xi^{p-1} \quad \hbox{for all}\ \xi\in\R.
    \end{equation}
We fix a $a_0\in (0,\half\underline V)$ and compute
    \begin{eqnarray*}
        \Ie'(u)u
        &=& \int_{\R^N}|(\frac{1}{i}\nabla - A(\e x) ) u|^2  \, dx 
        +\int_{\R^N} V(\epsilon x)|u|^2\, dx -\int_{\R^N}f(|u|^2)|u|^2\, dx \\
        &\geq& \int_{\R^N}|(\frac{1}{i} \nabla - A(\e x) ) u|^2\,dx 
        +\underline V\norm u_2^2 -a_0\norm u_2^2-C_{a_0}\norm u_{p+1}^{p+1}\\
        &\geq& \int_{\R^N}|(\frac{1}{i} \nabla - A(\e x) ) u|^2\,dx 
        +\half \underline V\norm u_2^2 -C_{a_0}\norm u_{p+1}^{p+1}.
    \end{eqnarray*}
Now choosing $r_2>0$ small enough there exists $c>0$ such that
    \be\label{eq:3.2}
        \int_{\R^N}|(\frac{1}{i}\nabla - A (\e x) ) u|^2 \, dx
        +\half \underline V\norm u_2^2 -2^p C_{a_0}\norm u_{p+1}^{p+1}
        \geq c\norm{u}_{\e}^2 \quad \hbox{for all}\quad \norm u_{\e}\leq 2r_2.
    \ee
(For a technical reason, especially to get \eqref{eq:3.23} later, we add \lq\lq $2^p$'' in front of
$C_{a_0}$.)
In particular, we have
    \be\label{eq:3.3}
        \Ie'(u)u \geq c \norm{u}_{\e}^2 \quad \hbox{for all}\quad  \norm{u}_{\e} \leq 2r_2.
    \ee
Now we set
    $$  n_\epsilon =\left[{h_0\over \epsilon}\right]-1
    $$
and for each $i=1,2,\cdots,n_\epsilon$ we fix a function $\varphi_{\epsilon,i}\in C_0^\infty(\Omega)$
such that
    \begin{eqnarray*}
        &&\varphi_{\epsilon,i}(x)=\begin{cases}    1   &\hbox{if $x\in\Omega_{\epsilon,i}$,} \\
                                            0   &\hbox{if $x\not\in\Omega_{\epsilon,i+1}$,} \end{cases}\\
        &&\varphi_{\epsilon,i}(x)\in [0,1], \ \abs{\varphi'_{\epsilon,i}(x)}\leq 2
            \quad \hbox{for all}\ x\in\Omega.
    \end{eqnarray*}
Here we denote for $\epsilon>0$ and $h\in (0,2h_0/\epsilon]$
    \begin{eqnarray*}
        \Omega_{\epsilon,h} &=& (\Omega_{\epsilon h})/\epsilon \\
        &=&\{ x\in \R^N\setminus(\Omega/\epsilon);\,
        \dist(x,(\partial\Omega)/\epsilon)<h\}\cup (\Omega/\epsilon).
    \end{eqnarray*}
Now suppose that a sequence $(\ue)\subset\calS_\e(r_2)$ satisfies
for $0<\rho_1<\rho_0<r_2$
    \begin{eqnarray}
        &&\Je(\ue) \leq E(m_0+\nu_1),                        \label{eq:3.4}\\
        &&\wrho_\e(\ue) \in [0,\rho_0],                         \label{eq:3.5}\\
        &&\epsilon\Upsilon_\e(\ue)\in\Omega([0,\nu_0]),         \label{eq:3.6}\\
        &&\norm{\Je'(\ue)}_{(H_\e)^*} \to 0.                   \label{eq:3.7}
    \end{eqnarray}
We shall prove, in several steps, that for $\epsilon>0$ small
    \be\label{eq:3.8}
        \wrho_\e(\ue)\in [0,\rho_1]\quad \hbox{and}\quad \epsilon\Upsilon_\e(\ue)\in \Omega([0,\nu_1]),
    \ee
from which the conclusion of Proposition \ref{claim:3.1} follows.
\vspace{2mm}

\noindent
{\sl Step 1: There exists a $i_\epsilon\in \{ 1,2,\cdots, n_\epsilon\}$ such that
    \be\label{eq:3.9}
        \norm{\ue}_{H_\e(\Omega_{\epsilon,i_\epsilon+1}\setminus\Omega_{\epsilon,i_\epsilon})}^2
        \leq {4r_2^2\over n_\epsilon}.
    \ee
Here we use notation:
    $$  \| u\|_{H_\e(K)}^2=\int_K \left|\left(\frac{1}{i}\nabla -A_\e\right)u\right|^2+|u|^2\, dx
    $$
for $u\in H_\e$ and $K\subset \R^N$.
}

\noindent
Indeed we can write 
$u_\e = e^{i A(\e p_\e) (x- p_\e)} U_\e(x - p_\e) + \varphi_\e(x)$ with $\| \varphi_\e\|_\e \leq r_2$ and
by \eqref{eq:3.6} and the uniform exponential decay of $\whS$, we have
    $$
    \norm{\ue}_{H_\e(\R^N\setminus(\Omega/\epsilon))} \leq
    \norm{e^{i A(\e p_\e) (x- p_\e)} U_\e(x - p_\e)}_{H_\e(\R^N\setminus(\Omega/\epsilon))} 
    + \norm{\varphi_\e}_{H_\e(\R^N\setminus(\Omega/\epsilon))} \leq 2r_2
    $$
for $\epsilon >0$ small.  Thus
    $$  \sum_{i=1}^{n_\epsilon}\norm{\ue}_{H_\e(\Omega_{\epsilon,i+1}\setminus\Omega_{\epsilon,i})}^2
        \leq \norm{\ue}_{H_\e(\Omega_{\epsilon,h_0/\epsilon}\setminus(\Omega/\epsilon))}^2
        \leq 4r_2^2
    $$
and there exists $i_\epsilon\in \{ 1,2,\cdots, n_\epsilon\}$ satisfying \eqref{eq:3.9}.

\medskip

\noindent
{\sl Step 2: For the $i_\epsilon$ obtained in Step 1, we set
$$
        \ue^{(1)}(x) = \varphi_{\epsilon,i_\epsilon}(x)\ue(x) \quad \mbox{and} \quad
        \ue^{(2)}(x) = \ue(x)-\ue^{(1)}(x).  $$
Then we have, as $\epsilon\to 0$,
    \begin{eqnarray}
        &&\Ie(\ue^{(1)})=\Je(\ue)+o(1),                  \label{eq:3.10}\\
        &&\|\Ie'(\ue^{(1)})\|_{(H_\e)^*}\to 0,      \label{eq:3.11}\\
        &&\|\ue^{(2)}\|_\e \to 0,     \label{eq:3.12}\\
        &&\Qe(\ue^{(2)})\to 0.                           \label{eq:3.13}
    \end{eqnarray}}

\smallskip

\noindent
Observe that
 \be\label{eq:3.14}
        \Ie(\ue) = \Ie(\ue^{(1)})+\Ie(\ue^{(2)}) +o(1).
    \ee
Indeed, by \eqref{eq:3.9}
    \begin{eqnarray*}
    &&\Ie(\ue)-(\Ie(\ue^{(1)})+\Ie(\ue^{(2)}))  \\
    &=&  Re \int_{\Omega_{\epsilon,i_\epsilon+1}\setminus\Omega_{\epsilon,i_\epsilon}}
        (\frac{1}{i}\nabla -A(\e x))(\varphi_{\epsilon,i_\epsilon}\ue) 
            \overline {\left( \frac{1}{i}\nabla - A(\e x)\right)((1-\varphi_{\epsilon,i_\epsilon})\ue)} \\
        &&
        +V(\epsilon x) \varphi_{\epsilon,i_\epsilon}(1-\varphi_{\epsilon,i_\epsilon})|\ue|^2\, dx\\
       & &- \half\int_{\Omega_{\epsilon,i_\epsilon+1}\setminus\Omega_{\epsilon,i_\epsilon}}
            F(|\ue|^2)-F(|\ue^{(1)}|^2)-F(|\ue^{(2)}|^2)\, dx \\
    &\to& 0 \qquad \hbox{as}\ \epsilon\to 0.
    \end{eqnarray*}

\noindent
Thus
    \be\label{eq:3.15}
        \Je(\ue) = \Ie(\ue^{(1)})+\Ie(\ue^{(2)}) +\Qe(\ue^{(2)})+o(1).
    \ee
We can also see that
    \be\label{eq:3.16}
        \norm{\Ie'(\ue)-\Ie'(\ue^{(1)}) -\Ie'(\ue^{(2)})}_{(H_\e)^*} \to 0
        \quad \hbox{as}\ \epsilon\to 0.
    \ee
In a similar way, it follows from \eqref{eq:3.7} that, since $\|\ue^{(2)}\|_\e$ is bounded, that
    \be\label{eq:3.17}
        \Ie'(\ue^{(2)})\ue^{(2)} + \Qe'(\ue^{(2)})\ue^{(2)}
        = \Je'(\ue)\ue^{(2)} +o(1) = o(1).
    \ee
We note that $\norm{\ue^{(2)}}_{\e} \leq 2r_2$ and $(p+1)\Qe(u) \leq \Qe'(u)u$ for all
$u\in H_\e$.  Thus by \eqref{eq:3.3}
    $$  c\norm{\ue^{(2)}}_{\e}^2 + (p+1)\Qe(\ue^{(2)})\to 0 \quad
        \hbox{as}\ \epsilon\to 0,
    $$
which implies \eqref{eq:3.12} and \eqref{eq:3.13}. Now  \eqref{eq:3.12} implies that
$\Ie(\ue^{(2)})\to 0$ and thus \eqref{eq:3.10} follows from  \eqref{eq:3.15}.
\vspace{2mm}

Finally we show \eqref{eq:3.11}.  We choose a function $\widetilde\varphi\in C_0^\infty(\R^N)$
such that
    $$  \widetilde\varphi(x)=\begin{cases}    1   &\hbox{for $x\in \Omega_{h_0}$,}\\
                                        0   &\hbox{for $x\in \R^N\setminus\Omega_{2h_0}$.}
                            \end{cases}
    $$
Then we have, for all $w\in H_\e$,
    \begin{eqnarray*}
    \Ie'(\ue^{(1)})w &=& \Ie'(\ue^{(1)})(\widetilde\varphi(\epsilon x)w) \\
    &=& \Ie'(\ue)(\widetilde\varphi(\epsilon x)w)
        - (\Ie'(\ue)-\Ie'(\ue^{(1)}))(\widetilde\varphi(\epsilon x)w)\\
    &=& \Je'(\ue)(\widetilde\varphi(\epsilon x)w)
        - (\Ie'(\ue)-\Ie'(\ue^{(1)}))(\widetilde\varphi(\epsilon x)w)
    \end{eqnarray*}
and it follows that
    $$  \abs{\Ie'(\ue^{(1)})w}
        \leq \norm{\Je'(\ue)}_{(H_\e)^*}\norm{\widetilde\varphi(\epsilon x)w}_{\e}
        +\norm{\Ie'(\ue)-\Ie'(\ue^{(1)})}_{(H_\e)^*}\norm{\widetilde\varphi(\epsilon x)w}_{\e}.
    $$
We note that by \eqref{eq:3.12} and \eqref{eq:3.16}, $\norm{\Ie'(\ue)-\Ie'(\ue^{(1)})}_{(H_\e)^*}
\to 0$.  Therefore, by \eqref{eq:3.7}, $\norm{\Ie'(\ue^{(1)})}_{(H_\e)^*} \to 0$, that is
\eqref{eq:3.11} holds true.

\medskip

\noindent
{\sl Step 3: After extracting a subsequence --- still we denoted by $\epsilon$ ---,
there exist a sequence $(\tilde p_\e)\subset\R^N$ and $\tilde U\in\whS$ such that
    \begin{eqnarray}
    &&\epsilon \tilde p_\e\to \tilde p_0\quad \hbox{for some}\ \tilde p_0\in \Omega([0,\nu_1]),
                                                                        \label{eq:3.18}\\
    &&\norm{\ue^{(1)}- e^{i A(\e \tilde p_\e) (x-\tilde p_\e)} \tilde U(x-\tilde p_\e)}_{\e}\to 0,                       \label{eq:3.19}\\
    &&\Ie(\ue^{(1)})\to L_{V(\tilde p_0)}(\tilde U) \quad \hbox{as}\ \epsilon\to 0.    \label{eq:3.20}
    \end{eqnarray}
}
\noindent
Let $q_\epsilon=\Upsilon_\e(\ue)$.  We may assume that
    \begin{equation}\label{eq:3.21}
    e^{-i A(\e q_\e) (x + q_\e)} \ue^{(1)}(x+q_\epsilon)\wlimit U(x) \quad \hbox{weakly in}\ H^1(\R^N)
    \end{equation}
for some $U \in H^1(\R^N) \setminus\{ 0\}$ and also that $\epsilon q_\epsilon\to q_0\in\overline\Omega$.
In fact, noting $A_\e(x)$ is uniformly bounded on $\supp u^{(1)}_\e\subset\Omega_{\e, n_\e+1}\subset \Omega_{h_0}/\e$,
boundedness of $\|u^{(1)}_\e\|_\e$ implies boundedness of $\|u^{(1)}_\e\|_{H^1}$.  Thus, 
$e^{-i A(\e q_\e) (x + q_\e)} \ue^{(1)}(x+q_\epsilon)$  is also bounded in $H^1(\R^N)$.  
Taking a subsequence if necessary, we have \eqref{eq:3.21}.

From the definition of $\Upsilon_\e$ and \eqref{eq:3.11}, it follows that
    $(L^{\C}_{V(q_0)})'(U)=0$ and $U\not=0$.
In particular, $U(x)$ decays exponentially as $|x|\to\infty$.
Setting
    $$  w_\e (x)=   \ue^{(1)}(x+q_\epsilon)-  e^{i A(\e q_\e) \, x } U(x),
    $$
we have $w_\e $ is bounded in $H^1(\R^N)$ and 
    $   e^{-i A(\e q_\e) (x + q_\e)}w_\e  \wlimit 0 \quad \text{weakly in}\ H^1(\R^N)
    $
and thus $w_\e \wlimit 0$ weakly in $H^1(\R^N)$.  
We shall prove that $\norm{w_\e (x-q_\e)}_{\e}\to 0$.   We have from the exponential decay of
$U(x)$ that
\begin{eqnarray}
    &&\Ie'(\ue^{(1)})w_\e (x-q_\e)\nonumber\\
    &=&\ \Ie'(e^{i A(\e q_\e) (x - q_\e)}  U(x-q_\e)
        +w_\e (x-q_\e))w_\e (x-q_\e) \nonumber\\
    &=&\ {\rm Re}\int_{\R^N} \left(\frac{1}{i} \nabla - A(\e x + \e q_\e)\right) (e^{i A(\e q_\e) \, x } 
    U+w_\e ) 
    \overline{ \left(\frac{1}{i} \nabla - A(\e  x + \e q_\e)\right)w_\e }\, dx \nonumber\\
    &&\ +{\rm Re}\int_{\R^N}
        V(\e x+\e q_\e)(e^{i A(\e q_\e) \, x } U+w_\e )\overline{w_\e }\, dx \nonumber\\
    &&\    -{\rm Re}\int_{\R^N} f(|e^{i A(\e q_\e)x} U+w_\e |^2)
        (e^{i A(\e q_\e)x}U+w_\e )\overline{w_\e }\,dx\nonumber\\
    &=&\ \int_{\R^N} |\left(\frac{1}{i} \nabla - A(\e x + \e q_\e)\right)w_\e |^2 
    +V(\e x+\e q_\e)|w_\e |^2\,dx\nonumber\\
    &&\ +{\rm Re}\int_{\R^N} \left(\frac{1}{i}\nabla-A(\e x+\e q_\e)\right)(e^{iA(\e q_\e)x} U)
        \overline{\left(\frac{1}{i}\nabla-A(\e x+\e q_\e)\right)w_\e }\, dx \nonumber\\
    &&\ +{\rm Re}\int_{\R^N} V(\e x+\e q_\e)
        e^{iA(\e q_\e)x} U\overline{w_\e }\, dx\nonumber\\
    &&\ -{\rm Re}\int_{\R^N} f(|e^{i A(\e q_\e)x}U+w_\e |^2)
        (e^{i A(\e q_\e)x}U+w_\e )\overline{w_\e }\,dx\nonumber\\
    &=&\ \int_{\R^N} |\left(\frac{1}{i} \nabla - A(\e x + \e q_\e)\right)w_\e |^2 
    +V(\e x+\e q_\e)|w_\e |^2\,dx\nonumber\\
    &&\ +{\rm Re}\int_{\R^N} \left(\frac{1}{i}\nabla-A(\e q_\e)\right)(e^{iA(\e q_\e)x} U)
        \overline{\left(\frac{1}{i}\nabla-A(\e q_\e)\right)w_\e }\, dx \nonumber\\
    &&\ +{\rm Re}\int_{\R^N} V(\e q_\e)
        e^{iA(\e q_\e)x} U\overline{w_\e }\, dx +o(1)\nonumber\\
    &&\ -{\rm Re}\int_{\R^N} f(|e^{i A(\e q_\e)x}U+w_\e |^2)
        (e^{i A(\e q_\e)x}U+w_\e )\overline{w_\e }\,dx\nonumber\\
    &=&\ \int_{\R^N} |\left(\frac{1}{i} \nabla - A(\e x + \e q_\e)\right)w_\e |^2 
    +V(\e x+\e q_\e)|w_\e |^2\,dx\nonumber\\
    &&\ + (L^{\C}_{V(\e q_\e)})'(U)(e^{-iA(\e q_\e)x} w_\e ) 
        +{\rm Re}\int_{\R^N} f(|U|^2)U\overline{e^{-i A(\e q_\e)x}w_\e }\,dx\nonumber\\
    &&\ -{\rm Re}\int_{\R^N} f(|e^{i A(\e q_\e)x}U+w_\e |^2)
        (e^{i A(\e q_\e)x}U+w_\e )\overline{w_\e }\,dx+o(1)\nonumber\\
    &=&\ \int_{\R^N} |\left(\frac{1}{i}\nabla - A(\e x + \e q_\e)\right)w_\e |^2 
    +V(\e x+\e q_\e)w_\e ^2\,dx \nonumber\\
    &&\ +(L^{\C}_{V(\e q_\e)})'(U)(e^{-iA(\e q_\e)x} w_\e ) 
        +(I)-(II)+o(1).                              \label{eq:3.22}
    \end{eqnarray}
Since $(L^{\C}_{V(\e q_\e)})'(U)\to (L^{\C}_{V(p_0)})'(U)=0$, we have
    $$  {(L^{\C}_{V(\e q_\e)})}'(U)(e^{-iA(\e q_\e)x}w_\e ) \to 0.
    $$
Now, by \eqref{eq:3.1}, 
    \begin{eqnarray*}
    \abs{(I)}+\abs{(II)} &\leq& \int_{\R^N} (a_0(\abs{U}+\abs{e^{iA(\e q_\e)x}U+w_\e })
        +C_{a_0}(\abs{U}^p+\abs{e^{iA(\e q_\e)x}U+w_\e }^p))
        \abs{w_\e }\, dx\\
    &\leq& \int_{\R^N} a_0\abs{w_\e }^2 + 2^pC_{a_0}\abs{w_\e }^{p+1}
        +(2a_0\abs{U}+(1+2^p)C_{a_0}\abs{U}^p)\abs{w_\e }
        \, dx\\
    &\leq& \int_{\R^N} a_0\abs{w_\e }^2 + 2^pC_{a_0}\abs{w_\e }^{p+1}
        \, dx + o(1).
    \end{eqnarray*}
Here we used the fact that $w_\e \wlimit 0$ weakly in $H^1(\R^N)$.
Thus, by \eqref{eq:3.22} and \eqref{eq:3.11}, we have
    $$  \int_{\R^N} |\left(\frac{1}{i} \nabla - A(\e x + \e q_\e )\right)w_\e |^2 +\underline V\norm{w_\e }_2^2
        \leq a_0 \|w_\e (x-q_\e) \|_\e^2 + 2^p C_{a_0} \norm{w_\e }_{p+1}^{p+1}
        +o(1)
    $$
from which we deduce, using \eqref{eq:3.2}, that
\be\label{eq:3.23}
\norm{w_\e (x-q_\e)}_{\e}\to 0.
\ee
At this point we have obtained  \eqref{eq:3.19}, \eqref{eq:3.20} where $\tilde p_\e$, $\tilde p_0$ and 
$\tilde U$ are replaced  with $q_\epsilon$, $q_0$ and $U$.
Since
    \be\label{eq:3.24}
        \Ie(\ue) = \Ie(\ue^{(1)}) +o(1) = \Je(\ue) + o(1) \leq E(m_0+\nu_1)+o(1)
    \ee
implies
    $$  E(V(q_0)) \leq L_{V(q_0)}(U) \leq E(m_0+\nu_1),
    $$
we have $\tilde p_0=q_0\in \Omega([0,\nu_1])$ and $U$ belongs to $S_{V(\tilde p_0)}\subset\whS$ after
a suitable shift, that is, $U(x):=\widetilde U(x+y_0)\in \whS$ for some $y_0\in\R^N$.
Setting $\tilde p_\e=q_\epsilon+y_0$, we get \eqref{eq:3.18}--\eqref{eq:3.20}.

\medskip

\noindent
{\sl Step 4: Conclusion}

\smallskip

\noindent
In Steps 1--3, we have shown that a sequence $(u_n) \subset \calS_\e(r_2)$ with \eqref{eq:3.4}--\eqref{eq:3.7} 
satisfies, up to a subsequence, and for some $U\in\whS$  \eqref{eq:3.19}--\eqref{eq:3.20} with 
$\tilde p_{\e} = \Upsilon_\e(\ue) + y_0$. This implies that
    \begin{eqnarray*}
        &\epsilon\Upsilon_\e(\ue)\to \tilde p_0\in \Omega([0,\nu_1]),\\
        &\norm{\ue(x)-e^{i A(\e \tilde p_\e) \, (x- \tilde p_\e) }U(x-\tilde p_\e)}_{\e} \to 0.
    \end{eqnarray*}
In particular since $\wrho_\e(\ue)\to 0$ as $\e\to 0$, we have $\wrho_\e(\ue) \in [0, \rho_1] $ 
and \eqref{eq:3.8} holds.
This ends the proof of the Proposition.  
\end{proof}

\medskip

\begin{proposition}\label{claim:3.2}
There exists $\epsilon_0>0$ such that for $\epsilon\in (0,\epsilon_0]$ if $\ue\in\calS_\e(r_2)$
satisfies
    \begin{eqnarray}
        &&\Je'(\ue)=0,                       \label{eq:3.25}\\
        &&\Je(\ue)\leq E(m_0+\nu_1),         \label{eq:3.26}\\
        &&\epsilon\Upsilon_\e(\ue)\in\Omega([0,\nu_0]), \label{eq:3.27}
    \end{eqnarray}
then
    \be\label{eq:3.28}
        \Qe(\ue)=0 \quad \hbox{and}\quad \Ie'(\ue)=0.
    \ee
That is, $\ue$ is a solution of \eqref{eq:2.1}.
\end{proposition}

\medskip

\begin{proof}
Suppose that $\ue$ satisfies \eqref{eq:3.25}--\eqref{eq:3.27}.
Since $\ue$ satisfies \eqref{eq:3.25}
we have    \begin{eqnarray}
        &&\bigg( \frac{1}{i} \nabla - A_\e \bigg)^2  u_\e + \Bigl(V_\epsilon+(p+1)
        (\epsilon^{-2}\norm{\ue}_{L^2(\R^N\setminus(\Omega_{2h_0}/\epsilon))}^2-1)_+^{p-1\over 2} \nonumber\\
        &&\qquad \qquad\ \times \epsilon^{-2} \chi_{\R^N\setminus(\Omega_{2h_0}/\epsilon)}(x)\Bigr) \ue
         = f(|\ue|^2) \ue,                                  \label{eq:3.29}
    \end{eqnarray}
where $\chi_{\R^N\setminus(\Omega_{2h_0}/\epsilon))}(x)$ is the characteristic function of the set
${\R^N\setminus(\Omega_{2h_0}/\epsilon)}$.

Clearly $\ue$ satisfies \eqref{eq:3.4}--\eqref{eq:3.7} and thus, by the proof of Proposition \ref{claim:3.1}, 
we have
    $$  \norm{\ue}_{H_\e(\R^N\setminus(\Omega_{h_0}/\epsilon))}
        \leq \norm{\ue^{(2)}}_{\e}\to 0  \quad \hbox{as}\ \epsilon\to 0.
    $$
From Moser's iteration scheme, it follows that
    $$  \norm{\ue}_{L^\infty(\R^N\setminus(\Omega_{{3\over 2}h_0}/\epsilon))}\to 0
        \quad \hbox{as}\ \epsilon\to 0
    $$
and using a comparison principle, we deduce that for some $c$, $c'>0$
    $$  \abs{\ue(x)} \leq c'\exp(-c\dist(x,\Omega_{{3\over 2}h_0}/\epsilon)).
    $$
In particular then
    $$  \norm{\ue}_{L^2(\R^N\setminus(\Omega_{2h_0}/\epsilon))}<\epsilon \quad
        \hbox{for} \quad  \epsilon >0 \ \hbox{small}\
    $$
and we have \eqref{eq:3.28}. 
\end{proof}

\vspace{2mm}

To find critical points of $\Je$, we need the following.

\medskip

\begin{proposition} \label{claim:3.3}
For any fixed $\epsilon>0$, the Palais-Smale condition holds for $\Je$ in $\{ u\in\calS_\e(r_2);\,
\epsilon\Upsilon_\e(u)\in\Omega([0,\nu_0])\}$.  That is, if a sequence $(u_j)\subset H_\e$
satisfies for some $c>0$
    \begin{eqnarray*}
        &&u_j \in\calS_\e(r_2),\\
        &&\epsilon\Upsilon_\e(u_j) \in \Omega([0,\nu_0]),\\
        &&\norm{\Je'(u_j)}_{(H_\e)^*} \to 0,\\
        &&\Je(u_j)\to c \quad \hbox{as}\ j\to \infty,
    \end{eqnarray*}
then $(u_j)$ has a strongly convergent subsequence in $H_\e$.
\end{proposition}

\medskip

\begin{proof}
Since $\calS_\e(r_2)$ is bounded in $H_\e$, after extracting a subsequence if necessary, we
may assume $u_j\wlimit u_0$ weakly in $H_\e$ for some $u_0\in H_\e$.  We will
show that $u_j\to u_0$ strongly in $H_\e$.
Denoting $B_R=\{ x\in\R^N;\, \abs x<R\}$, it suffices to show that
    \be\label{eq:3.31}
        \lim_{R\to\infty} \lim_{j\to\infty} \norm{u_j}_{H_\e(\R^N\setminus B_R)}^2=0.
    \ee
To show \eqref{eq:3.31} we first we note that, since $\varepsilon >0$ is fixed,  
$\norm{u_j}_{H_\e(\R^N\setminus B_L)}<2r_2$ for
a large $L>1$.  In particular, for any $n\in\N$
    $$  \sum_{i=1}^n \norm{u_j}_{H_\e(D_i)}^2 < 4r_2^2,
    $$
where $D_i= B_{L+i}\setminus B_{L+i-1}$.

Thus, for any $j\in\N$, there exists $i_j\in \{1,2,\cdots,n\}$ such that
    $$  \norm{u_j}_{H_\e(D_{i_j})}^2 < {4r_2^2\over n}.
    $$
Now we choose $\zeta_i\in C^1(\R,\R)$ such that $\zeta_i(r)=1$ for $r\leq L+i-1$,
$\zeta_i(r)=0$ for $r\geq L+i$ and $\zeta_i'(r)\in [-2,0]$ for all $r>0$.
We set
    $$
     \widetilde u_j(x)=(1- \zeta_{i_j}(\abs x))u_j(x).
    $$
We have, for a constant  $C>0$ independent of $n$, $j$
    \begin{eqnarray}
& &    \Je'(u_j)\widetilde u_j = \Ie'(u_j)\widetilde u_j+\Qe'(u_j)\widetilde u_j,    \nonumber \\
& &    \Ie'(u_j)\widetilde u_j = \Ie'(\widetilde u_j)\widetilde u_j
       +{\rm Re}\int_{D_{i_j}}  \left(\frac{1}{i} \nabla- A(\e x)\right)(\zeta_{i_j}u_j)
        \overline{\left(\frac{1}{i} \nabla- A(\e x)\right)((1-\zeta_{i_j})u_j)}\, dx \nonumber \\ 
        && + \int_{D_{i_j}} V(\epsilon x)\zeta_{i_j}(1- \zeta_{i_j})|u_j|^2  + 
        [f(|(1- \zeta_{i_j})u_j|^2)(1-\zeta_{i_j}) - f(|u_j|^2)](1-\zeta_{i_j})|u_j|^2\, dx \nonumber\\
   & & \geq \Ie'(\widetilde u_j)\widetilde u_j -{C\over n}, \label{key} \\
    & & \Qe'(u_j)\widetilde u_j =
    (p+1)\left(\epsilon^{-2}\norm{u_j}_{L^2(\R^N\setminus(\Omega_{2h_0}/\epsilon))}^2-1\right)_+^{p-1\over 2} \nonumber \\ &&
   \times \e^{-2}\int_{\R^N\setminus(\Omega_{2h_0}/\epsilon))}(1-\zeta_{i_j})|u_j|^2\, dx \geq 0.  \label{eq:3.33}  
    \end{eqnarray}
Since $\Je'(u_j)\widetilde u_j\to 0$, it follows from \eqref{key}--\eqref{eq:3.33} that
    $$  \Ie'(\widetilde u_j)\widetilde u_j \leq {C\over n}+o(1) \quad
        \hbox{as}\ j\to \infty.
    $$
Now recording that $\norm{\widetilde u_j}_{\e} < 2 r_2$ we have by  \eqref{eq:3.3} for some $C>0$
 $$
\norm{\widetilde u_j}_{\e}^2 \leq {C\over n}+ o(1).
 $$
Thus, from the definition of $\widetilde u_j$, we deduce that
$$
\norm{u_j}_{H_\e(\R^N\setminus B_{L+n})}^2 \leq {C\over n}+o(1).
$$
That is, \eqref{eq:3.31} holds and $(u_j)$ strongly converges.
\end{proof}

\bigskip

The following lemma will be useful to compute the relative category.

\medskip

\begin{lemma}\label{claim:3.4}
There exists $C_0>0$ independent of $\epsilon >0$ such that
    \be\label{eq:3.34}
        \Je(u) \geq L_{m_0}(|u|) -C_0\epsilon^2 \quad \hbox{for all}\ u\in\calS_\e(r_1).
         \ee
\end{lemma}

\begin{proof}
    \begin{eqnarray*}
        \Je(u) &\geq& L_{m_0}(|u|)+\half\intRN(V(\epsilon x)-m_0)|u|^2\, dx +Q_\epsilon(u)\\
        &\geq& L_{m_0}(|u|) -\half(m_0-\underline V)\norm u_{L^2(\R^N\setminus(\Omega/\epsilon))}^2
                +Q_\epsilon(u).
    \end{eqnarray*}
We distinguish the two cases:  (a) $\norm u_{L^2(\R^N\setminus(\Omega/\epsilon))}^2\leq 2\epsilon^2$,
(b) $\norm u_{L^2(\R^N\setminus(\Omega/\epsilon))}^2\geq 2\epsilon^2$.
\vspace{1mm}

If case (a) occurs, we have
    $$  \Je(u) \geq L_{m_0}(|u|)-(m_0-\underline V)\epsilon^2
    $$
and \eqref{eq:3.34} holds.  If case (b) takes  place, we have
    $$  Q_\epsilon(u)
        \geq \left(\half\epsilon^{-2}\norm u_{L^2(\R^N\setminus(\Omega/\epsilon))}^2
                \right)^{p+1\over 2}
        \geq \half\epsilon^{-2}\norm u_{L^2(\R^N\setminus(\Omega/\epsilon))}^2
    $$
and thus
    \begin{eqnarray*}
    \Je(u) &\geq& L_{m_0}(|u|)
    +\half(\epsilon^{-2}-(m_0-\underline V))\norm u_{L^2(\R^N\setminus(\Omega/\epsilon))}^2\\
    &\geq& L_{m_0}(|u|) \quad \hbox{for $\epsilon>0$ small}.
    \end{eqnarray*}
Therefore \eqref{eq:3.34} also holds. 
\end{proof}

\medskip

\setcounter{equation}{0}
\section{\label{section:4} A $\mathbb{S}^1$-invariant neighborhood of expected solutions}

In order to find critical points of the penalized functional $\Je$, we need to find a
$\mathbb{S}^1$-invariant neighborhood  $\calXed$ of expected solutions, which is  positively invariant under a
$\mathbb{S}^1$-equivariant pseudo-gradient flow.

\medskip

We fix $0<\rho_1<\rho_0<r_2$ and we then choose $\delta_1$, $\delta_2>0$ according to Lemma \ref{claim:2.7}
and Proposition \ref{claim:3.1}.
We set for $\delta\in (0,\min\{ {\delta_2\over 4}(\rho_0-\rho_1),\delta_1\})$,
    $$  \calXed=\{ u\in\calS_\e(\rho_0);\, \epsilon\Upsilon_\e(u)\in\Omega([0,\nu_0]),\
        \Je(u)\leq E(m_0)+\delta-{\delta_2\over 2}(\wrho_\e(u)-\rho_1)_+\}.
    $$
We notice that $\calXed$ is $\mathbb{S}^1$-invariant, namely if $u \in \calXed$ then $ \gamma u \in \calXed$ 
for any $\gamma \in \mathbb{S}^1$.
We shall try to find $\mathbb{S}^1$-orbits of critical points of $\Je$ in $\calXed$. In this aim first note that

\begin{itemize}
\item[(a)] $u\in\calS_\e(\rho_0)$ and $\epsilon\Upsilon_\e(u)\in\Omega([\nu_1,\nu_0])$
imply, by Lemma \ref{claim:2.7}, that
    \be\label{eq:4.1}
        \Je(u) \geq \Ie(u) \geq E(m_0)+\delta_1 > E(m_0)+\delta.
    \ee

\noindent
In particular,
    $$   \epsilon\Upsilon_\e(u)\in\Omega([0,\nu_1)) \quad \hbox{for}\ u\in\calXed.
    $$
\item[(b)]  For $u\in\calXed$, if $\wrho_\e(u)=\rho_0$, i.e., $u\in\partial\calS_\e(\rho_0)$,
then by the choice of $\delta$
    \be\label{eq:4.2}
        \Je(u)\leq E(m_0)+\delta-{\delta_2\over 2}(\rho_0-\rho_1) < E(m_0)-\delta.
    \ee
\end{itemize}

\medskip

\subsection{\label{subsection:4.1} A $\mathbb{S}^1$-equivariant deformation theorem}

\noindent
Now we consider a deformation flow defined by
    \begin{equation}\label{eq:4.3}
    \begin{cases}
        {d\eta\over d\tau}=-\phi(\eta) {\calV(\eta)\over\norm{\calV(\eta)}_{H^1}},\\
        \eta(0,u)=u,
        \end{cases}
    \end{equation}
where $\calV(u):\, \{ u\in H_\e;\, \Je'(u)\not=0\}\to H_\e$ is a
locally Lipschitz continuous $\mathbb{S}^1$-equivariant, pseudo-gradient vector field satisfying
    $$  \norm{\calV(u)}_{\e}\leq \norm{\Je'(u)}_{(H_\e)^*},\quad
        \Je'(u)\calV(u)\geq \half\norm{\Je'(u)}_{(H_\e)^*}^2
    $$
and $\phi(u):\, H_\e\to [0,1]$ is a locally Lipschitz
continuous function. We require that $\phi(u) $ satisfies $\phi(u)=0$ if $\Je(u)\not\in
[E(m_0)-\delta,E(m_0)+\delta]$.

\medskip
Arguing as in \cite[Proposition 4.1]{CJT} (see also \cite[Theorem 1.8]{BBF}), we can derive 
the following deformation theorem in a neighborhood of expected solutions
$\calXed$.

\begin{proposition}\label{claim:4.1}
For any $c\in (E(m_0)-\delta,E(m_0)+\delta)$ and for any
$\mathbb{S}^1$-invariant neighborhood $O$ of
$\calK_c\equiv \{ u\in \calXed;\, \Je'(u)=0,\ \Je(u)=c\}$ ($O=\emptyset$ if
$\calK_c=\emptyset$), there exist $d>0$ and a
$\mathbb{S}^1$-equivariant
deformation
$\eta(\tau,u):\, [0,1]\times(\calXed\setminus O)\to\calXed$ such that
\begin{itemize}
\item[(i)] $\eta(0,u)=u$ for all $u$.
\item[(ii)] $\eta(\tau,u)=u$ for all $\tau\in [0,1]$ if $\Je(u)\not\in [E(m_0)-\delta,
E(m_0)+\delta]$.
\item[(iii)] $\Je(\eta(\tau,u))$ is a non-increasing function of $\tau$ for all $u$.
\item[(iv)] $\Je(\eta(1,u))\leq c-d$ for all $u\in \calXed\setminus O$ satisfying
$\Je(u)\leq c+d$.
\end{itemize}
\end{proposition}

\medskip

\subsection{\label{subsection:4.2} Two maps between topological pairs}

Now for $c\in\R$, we set
    $$  \calXed^c=\{ u\in\calXed;\, \Je(u)\leq c\}.
    $$
For $\hdelta>0$ small, using relative $\mathbb{S}^1$-equivariant category, we shall estimate the change of 
topology between $\calXp$ and $\calXm$.
\medskip

We recall that $K= \{ x \in \Omega ; V(x) = m_0\}. $ For $s_0\in (0,1)$ small we introduce
two maps:
    \begin{eqnarray*}
    && \tilde  \Phi_\epsilon:\, ([1-s_0,1+s_0]\times K, \{ 1\pm s_0\}\times K)
\to (\calXp, \calXm);\\
    && \tilde \Psi_\epsilon:\, (\calXp, \calXm) \to\\
    &&\qquad ([1-s_0,1+s_0]\times\Omega([0,\nu_1]), ([1-s_0,1+s_0]\setminus\{ 1\})\times\Omega([0,\nu_1]))
    \hss.
    \end{eqnarray*}
Here we use notation from algebraic topology: $f:\, (A,B)\to (A',B')$ means $B\subset A$,
$B'\subset A'$, $f:\, A\to A'$ is continuous and $f(B)\subset B'$.

\medskip

\noindent
{\sl Definition of $ \tilde\Phi_\epsilon$}: \\
Fix a least energy solution $U_0\in \whS$
of $-\Delta u+m_0 u=f(u)$ and set 
    $$
    \tilde \Phi_\epsilon(s,p)=   e^{i A(p) \left( {x-{p/ \epsilon}\over s} \right)} U_0
    \big(\frac{x- p/\epsilon}{s}\big).
    $$
 Let us show that $ \tilde \Phi_\epsilon$ is well-defined
for a suitable choice of $s_0$ and $\hdelta$ and assuming $\epsilon >0$ small enough.
\vspace{1mm}

By the exponential decay of $U_0$, we can find $s_0\in (0,1)$ small such that
    $$  \| e^{iA(p)(\frac{x-p/\e}{s})} U_0(\frac{x-p/\e}{s}) 
        - e^{iA(p)(x-p/\e)} U_0(x-p/\e)\|_\e <\rho_1
    $$
for all $p\in K$, $s\in [1-s_0,1+s_0]$ and small $\e>0$.
Therefore, using the first property of Lemma \ref{claim:2.6}, that is
$$|\Upsilon_\e(u)-p| \leq 2R_0$$ for
$u(x) = e^{i A(\e p) \left( {x- p} \right)} U(x-p) + \varphi(x) \in \calS_\e(\rho_0)$, we get
    $$
    \big|\Upsilon_\e(e^{i A(p) \left( \frac{x - p/\varepsilon}{s} \right)} 
    U_0\big(\frac{x - p/\varepsilon}{s}\big)) - p/\varepsilon \big| \leq 2 R_0.
    $$
It follows that, for $p\in K$, $s\in [1-s_0,1+s_0]$
\be\label{eq:add}
\epsilon\Upsilon_\e(e^{i A(p) \left( \frac{x - p/\varepsilon}{s} \right)}U_0\big({x-p/\epsilon\over s}\big))=p+o(1)
\ee
and so $\epsilon\Upsilon_\e \tilde \Phi_\epsilon(s,p) \in \Omega([0,\nu_0])$ for $\e >0$ small enough.

\vspace{1mm}

Since $U_0$ is a least energy solution, we note that $|U_0|$ satisfies the Pohozaev identity
\eqref{eq:limit} and thus $P_0(|U_0|)=1$ and $P_0(|U_0(\frac{x}{s})|)=s$.  
Also we have, by Lemma \ref{pohozaev}, for $p\in K$ and $s\in [1-s_0,1+s_0]$
    \begin{eqnarray*}
    \Je(e^{i A(p) \left(\frac{x-p/\epsilon}{s} \right)} U_0({x-p/\epsilon\over s}))
    &=& L_{m_0}(|U_0({x-p/\epsilon\over s})|)+o(1)\\
    &=& g(P_0(U_0({x-p/\epsilon\over s}))) E(m_0) +o(1)\\
    &=& g(s) E(m_0) + o(1),
    \end{eqnarray*}
where $g:\, [0,\infty)\to \R$ is defined in \eqref{eq:2.14}.
Note that $g$ satisfies $g(t)\leq 1$ for all $ t>0$ and that $ g(t)=1$ holds if and only if $t=1$. 
Thus choosing  $\hdelta>0$ small so that $g(1\pm s_0)E(m_0)<E(m_0)-\hdelta$, we have
$
\tilde \Phi_\epsilon(\{1\pm s_0\}\times K) \subset  \calXm
$
and thus $\tilde \Phi_\epsilon$ is well-defined. \medskip

\noindent
{\sl Definition of $\tilde \Psi_\epsilon$}:\\
We introduce the continuous function $\tilde P_0 : \calS_\e(r_0) \to \R$ by
    \begin{equation*}  \tilde P_0(u)=\begin{cases}
                            1+s_0 &\hbox{if $P_0(u)\geq 1+s_0$,}\\
                            1-s_0 &\hbox{if $P_0(u)\leq 1-s_0$,}\\
                            P_0(u)  &\hbox{otherwise,}
                        \end{cases}
    \end{equation*}
where $P_0$ is given in \eqref{eq:2.13} and we  define our operator $\tilde \Psi_\epsilon$ by
    $$ \tilde \Psi_\epsilon(u)=(\tilde P_0(u), \epsilon\Upsilon_\e(u))
        \quad \hbox{for}\ u\in\calXp.
    $$
Let us show that $\tilde \Psi_\epsilon$ is  well-defined for $\e>0$ small enough.
By  definition $\tilde \Psi_\epsilon(\calXp)\subset [1-s_0,1+s_0]\times
\Omega([0,\nu_1])$.  Now if $u\in\calXm$ we have by Lemma \ref{claim:3.4}, 
    $$
      L_{m_0}(|u|) \leq \Je(u) + C_0\epsilon^2 \leq E(m_0)-\hdelta + C_0\epsilon^2.
    $$
Thus, for $\e>0$ small enough, 
\begin{equation}\label{strict}
L_{m_0}(|u|) < E(m_0).
\end{equation}
At this point we recall, see \cite{JT1} for a proof, that $E(m_0)$ can be characterized as
    \be \label{eq:4.6}
      E(m_0)=\inf\{ L_{m_0}(u);\, u\not=0,\, P_0(u)=1\}.
    \ee
Thus \eqref{strict} implies that $P_0(u) \not= 1$ and $\tilde \Psi_\epsilon$ is well-defined. 
\vspace{2mm}

Now setting $\Phi_\epsilon(s,p): = \mathbb{S}^{1} \tilde \Phi_\e(s,p)$ for each 
$(s, p) \in [1-s_0,1+s_0]\times K$, it results that $\Phi_\epsilon$ is well-defined as a map
     \begin{eqnarray*}
    &&\Phi_\epsilon:\, ([1-s_0,1+s_0]\times K, \{ 1\pm s_0\}\times K)
        \to (\calXp/ \mathbb{S}^{1}, \calXm/ \mathbb{S}^{1}).\\
    \hss
    \end{eqnarray*}
Similarly setting$\Psi_\epsilon(\mathbb{S}^{1} u): = \tilde \Psi_\e(u)$ for any  $u \in \calXp$,  we see that
$\Psi_\epsilon$ is well-defined as a map
    \begin{eqnarray*}
    &&\Psi_\epsilon:\, (\calXp /\mathbb{S}^{1}, \calXm/\mathbb{S}^{1}) \to\\
    &&\qquad ([1-s_0,1+s_0]\times\Omega([0,\nu_1]), ([1-s_0,1+s_0]\setminus\{ 1\})\times\Omega([0,\nu_1])).
    \hss
    \end{eqnarray*}

\medskip

Finally we derive the following topological lemma.

\begin{proposition}\label{claim:4.2}
    \begin{eqnarray*}
    &&\Psi_\epsilon\circ\Phi_\epsilon:\,([1-s_0,1+s_0]\times K, \{ 1\pm s_0\}\times K) \\
    &&\quad\to([1-s_0,1+s_0]\times\Omega([0,\nu_1]),
    ([1-s_0,1+s_0]\setminus\{ 1\})\times\Omega([0,\nu_1]))
    \end{eqnarray*}
is homotopic to the embedding $j(s,p)=(s,p)$.  That is, there exists a continuous
map
    $$  \eta:\, [0,1]\times[1-s_0,1+s_0]\times K \to [1-s_0,1+s_0]\times\Omega([0,\nu_1])
    $$
such that
    \begin{eqnarray*}
        &&\eta(0,s,p)=(\Psi_\epsilon\circ\Phi_\epsilon)(s,p), \\
        &&\eta(1,s,p)=(s,p) \quad \hbox{for all}\ (s,p)\in [1-s_0,1+s_0]\times K,\\
        &&\eta(t,s,p)\in ([1-s_0,1+s_0]\setminus\{ 1\})\times\Omega([0,\nu_1])\\
        &&\qquad \hbox{for all}\ t\in [0,1] \ \hbox{and}\ (s,p)\in \{1\pm s_0\}\times K.
    \end{eqnarray*}
\end{proposition}

\begin{proof}
By the definitions of $\Phi_\epsilon$ and $\Psi_\epsilon$, we have
    \begin{eqnarray*}
    (\Psi_\epsilon\circ\Phi_\epsilon)(s,p)
    &=&\Bigl( \tilde P_0(\e^{i A(p) \left( \frac{x - p/ \e}{s} \right)} U_0({x-p/\epsilon\over s})),
        \epsilon\Upsilon_\e(e^{i A(p) \left( \frac{x - p/ \e}{s} \right)} U_0({x-p/\epsilon\over s}))\Bigr) \\
    &=& \Bigl( s, \epsilon\Upsilon_\e(e^{i A(p) \left( \frac{x - p/ \e}{s} \right)} U_0({x-p/\epsilon\over s}))\Bigr).
    \end{eqnarray*}
We set
    $$  \eta(t,s,p)=\Bigl( s, (1-t) \epsilon \Upsilon_\e(e^{i A(p) \left( \frac{x - p/ \e}{s} \right)} 
        U_0({x-p/\epsilon\over s})) +tp\Bigr).
    $$
Recalling \eqref{eq:add}, we see that for  $\epsilon>0$ small $\eta(t,s,p)$ has the desired properties and
$\Psi_\epsilon\circ\Phi_\epsilon$ is homotopic to the embedding $j$. 
\end{proof}

\medskip
\setcounter{equation}{0}
\section{\label{section:5} Proof of Theorem \ref{claim:1.1}}

In order to prove our theorem, we shall need some topological tools that we now present for the reader convenience.
Following \cite{BW}, see also \cite{FW1, FW2}, we define

\medskip

\begin{definition} \label{claim:5.1}
Let $B \subset A$ and $B' \subset A'$ be topological spaces and $f:(A,B) \to (A',B')$ be a continuous map, that is 
$f : A \to A'$ is continuous and $f(B) \subset B'$. The category $\cat(f)$ of $f$ is the least integer 
$k \geq 0$ such that there exist open sets $A_0$, $A_1$, $\cdots$, $A_k$ with the following properties:
\begin{itemize}
\item[(a)] $A=A_0\cup A_1\cup\cdots\cup A_k$.
\item[(b)] $B\subset A_0$ and there exists a map $h_0:\, [0,1]\times A_0\to A'$
such that
    \begin{eqnarray*}
        &&h_0(0,x)=f(x) \qquad \hbox{for all}\ x\in A_0,\\
        &&h_0(1,x)\in B' \quad \qquad \hbox{for all}\ x\in A_0,\\
        &&h_0(t,x)=f(x) \qquad \hbox{for all}\ x\in B\ \hbox{and}\ t\in [0,1].
    \end{eqnarray*}
\item[(c)] For $i=1,2,\cdots, k$, $f|_{A_i}:\, A_i\to A'$ is homotopic to a
constant map.
\end{itemize}
\end{definition}

We also introduce the cup-length of $f: (A,B) \to (A',B')$. Let $H^*$ denote Alexander-Spanier cohomology
with coefficients in the field $\F$.
We recall that the cup product $\smile$ turns $H^*(A)$ into a ring with unit $1_A$, and it turns $H^*(A,B)$
into a module over $H^*(A)$. A continuous map $f : (A,B) \to (A',B')$ induces a homomorphism
$f^* : H^*(A') \to H^*(A)$ of rings as well as a homomorphism $f^* : H^*(A',B') \to H^*(A,B)$ of
abelian groups.  We also use notation:
    $$  \tilde H^n(A')=\begin{cases} 0 &\text{for $n=0$,}\\ H^n(A')&\text{for $n>0$.}\end{cases}
    $$
For more details on algebraic topology we refer to \cite{S}.

\begin{definition} \label{claim:5.2}
For  $f:(A,B) \to (A',B')$ the cup-length, $\cuplength(f)$ is defined as follows;
when $f^*:\, H^*(A',B')\to H^*(A,B)$ is not a trivial map, $\cuplength(f)$ is defined
as the maximal integer $k\geq 0$ such
that there exist elements $\alpha_1$, $\cdots$, $\alpha_k
\in \tilde{H}^*(A')$  and $\beta\in H^*(A',B')$  with
    \begin{eqnarray*}  f^*(\alpha_1\smile\cdots\smile\alpha_k \smile \beta)
        &=& f^*(\alpha_1)\smile\cdots\smile f^*(\alpha_k)\smile f^*(\beta) \\
        &\not=& 0 \ \hbox{in}\ H^*(A,B).
    \end{eqnarray*}
When $f^*=0:\, H^*(A',B')\to H^*(A,B)$, we define $\cuplength(f)=-1$.
\end{definition}

We note that $\cuplength(f)=0$ if $f^*\not=0:\, H^*(A',B')\to H^*(A,B)$ and $\tilde{H}^*(A')=0$.


Finally we recall

\begin{definition} \label{claim:5.3}
For a set $(A,B)$, we define the relative category $\cat(A,B)$ and the relative
cup-length $\cuplength(A,B)$ by
    \begin{eqnarray*}
        &&\cat(A,B)=\cat(id_{(A,B)}:\, (A,B)\to (A,B)),\\
        &&\cuplength(A,B)=\cuplength(id_{(A,B)}:\, (A,B)\to (A,B)).
    \end{eqnarray*}
We also set
    $$  \cat(A)=\cat(A,\emptyset),\quad \cuplength(A)=\cuplength(A,\emptyset).
    $$
\end{definition}

\medskip

The following lemma which is due to Bartsch \cite{B} (see \cite{CJT} for a proof) is one of the keys of 
our proof and
we make use of the continuity property of Alexander-Spanier cohomology.

\medskip

\begin{lemma} \label{claim:5.4}
Let $K\subset \R^N$ be a compact set.  For a $d$-neighborhood $K_d=\{ x\in \R^N;\,
\dist(x,K)\leq d\}$ and $I=[0,1]$, $\partial I=\{0,1\}$, we consider the inclusion
    $$  j:\, (I\times K,\partial I\times K)\to (I\times K_d,\partial I\times K_d)
    $$
defined by $j(s,x)=(s,x)$.  Then for $d>0$ small,
    $$  \cuplength(j) \geq \cuplength(K).
    $$
\end{lemma}

Now we have all the ingredients to give the
\vspace{2mm}

\begin{proof}[Proof of Theorem \ref{claim:1.1}] %
Using Proposition \ref{claim:4.1}, we can apply $\mathbb{S}^1$-invariant L.S. theory and derive that
 for $\epsilon>0$ small the number of critical $\mathbb{S}^1$-orbits of $J_\e$ in $\calXp \setminus \calXm$ 
is at least  $\mathbb{S}^1$-$\cat(\calXp,\calXm)$   (see \cite[Theorem 4.2]{FW1} and \cite[Theorem 1.1]{cp}).

Since $\mathbb{S}^1$ acts freely on $H_\e \setminus \{0\}$, that is $\gamma u \neq u$ for all $u \in H_\e \setminus\{0\}$, $\gamma \in \mathbb{S}^1$, $\gamma \neq 1$, we have
\[
{\mathbb{S}}^1-\cat(\calXp,\calXm)= \cat (\calXp / \mathbb{S}^1, \calXm / \mathbb{S}^1).
\]
Finally, using Lemma \ref{claim:5.4}, we can argue as in the proof of Theorem 1.1 in \cite{CJT} and
we deduce that
    $$  \cat(\calXp/\mathbb{S}^1,\calXm/ \mathbb{S}^1) \geq \cuplength(K)+1.
    $$
Thus we conclude that $\Je$ has at least $\cuplength(K)+1$ critical $\mathbb{S}^1$ orbits in $\calXp\setminus\calXm$.
Recalling Proposition \ref{claim:3.2}, this completes the proof.  
\end{proof}

\bigskip

\begin{proof}[Proof of Remark \ref{mean-concen}]
From the proof of Proposition \ref{claim:3.1} we know that for any $\nu_0 >0$ small enough  
the critical points $u_{\e}^j$, $j= 1, \dots , \cuplength(K)+1$ satisfy
$$||u_{\e}^j (x) - e^{i A(x_j) (x- x_j) } U^j (x- x_\e^j) ||_{\e} \to 0$$
where $\e x_\e^j = \e \Upsilon_\e(u_{\e}^j) + o(1) \to x_0^j \in \Omega([0, \nu_0])$ and $U^i \in \whS$. 
Thus $w_{\e}^j(x) = u_{\e}^j(x + x_\e^j)$ converges to $ e^{i A(x_j) (x_j) } U^j \in \whS$. 
Now observing that these results holds for any $\nu_0 >0$ and any $\ell_0 > E(m_0+ \nu_0)$ we deduce, 
considering sequences $\nu_0^n \to 0$, $\ell_0^n \to E(m_0)$ and making a diagonal process, 
that it is possible to assume that each $w_{\e}^j$ converges to a least energy solution of
$$- \Delta U + m_0 U = f(U), \quad U >0, \quad U \in H^1(\R^N, \C).$$
Clearly also $$|u_{\e}^j(x)| \leq C exp(- c|x - x_{\e}^j|), \quad \mbox{for some } c, C >0$$
and this ends the proof. 
\end{proof}


\end{document}